\documentclass{amsart}
\usepackage{amsthm}
\newtheorem{theorem}{Theorem}[section]
\newtheorem{lemma}[theorem]{Lemma}
\newtheorem{proposition}[theorem]{Proposition}
\theoremstyle{definition}
\newtheorem{definition}[theorem]{Definition}

\theoremstyle{remark}
\newtheorem{remark}[theorem]{Remark}

\usepackage{amssymb}
\usepackage{empheq}
\usepackage{stmaryrd}
\usepackage{enumerate}
\usepackage[shortlabels]{enumitem}
\usepackage{calc}
\usepackage{url}

\usepackage{mathabx}

\usepackage{xcolor}

\newcommand{\norm}[1]{\left\lVert#1\right\rVert}
\usepackage{mathtools}
\DeclarePairedDelimiter{\ceil}{\lceil}{\rceil}
\DeclarePairedDelimiter{\floor}{\lfloor}{\rfloor}

\usepackage[left=2.0cm,%
                right=2.0cm,%
                top=2.5cm,%
                bottom=3.5cm,%
                headheight=12pt,%
                a4paper]{geometry}%

\begin{document}

\title{Generalized Fej\'er monotone sequences and their finitary content}

\author{Nicholas Pischke}
\date{\today}
\maketitle
\vspace*{-5mm}
\begin{center}
{\scriptsize Department of Mathematics, Technische Universit\"at Darmstadt,\\
Schlossgartenstra\ss{}e 7, 64289 Darmstadt, Germany, \ \\ 
E-mail: pischke@mathematik.tu-darmstadt.de}
\end{center}

\maketitle
\begin{abstract}
We provide quantitative and abstract strong convergence results for sequences from a compact metric space satisfying a certain form of \emph{generalized Fej\'er monotonicity} where (1) the metric can be replaced by a much more general type of function measuring distances (including, in particular, certain Bregman distances), (2) full Fej\'er monotonicity is relaxed to a partial variant and (3) the distance functions are allowed to vary along the iteration. For such sequences, the paper provides explicit and effective rates of metastability and even rates of convergence, the latter under a regularity assumption that generalizes the notion of metric regularity introduced by Kohlenbach, L\'opez-Acedo and Nicolae, itself an abstract generalization of many regularity notions from the literature. 

In the second part of the paper, we apply the abstract quantitative results established in the first part to two algorithms: one algorithm for approximating zeros of maximally monotone and maximally $\rho$-comonotone operators in Hilbert spaces (in the sense of Combettes and Pennanen as well as Bauschke, Moursi and Wang) that incorporates inertia terms \emph{every other} term and another algorithm for approximating zeros of monotone operators in Banach spaces (in the sense of Browder) that is only Fej\'er monotone w.r.t.\ a certain Bregman distance.
\end{abstract}

\section{Introduction}

This paper is about quantitative information on convergence results from analysis and optimization where the sequence in question satisfies a particular form of monotonicity, generalizing the well-known Fej\'er monotonicity (see \cite{BC2001,Com2009,VE2009} and see the next section for precise definitions).\footnote{Fej\'er monotonicity was introduced in \cite{MS1954}, named after the use made of a similar concept by L. Fej\'er \cite{Fej1922}.} As discussed e.g.\ in \cite{Com2009,VE2009}, many sequences studied in nonlinear analysis and optimization are Fej\'er or quasi-Fej\'er monotone (in the sense of \cite{Com2001}) and very general convergence theorems are available for these classes which classify additional properties required for strong or weak convergence (see \cite{Bro1967} for such results already implicitly and in particular see \cite{BC2001,Com2009}). Consequently, in the last decade(s), Fej\'er monotonicity has become a unifying concept for a wide range of convergence results in those areas. Particular examples include the generic example of Picard iterations for nonexpansive mappings but also the proximal point algorithm (stemming from \cite{Mar1970,Roc1976}), Mann iterations \cite{Man1953} as well as Ishikawa iterations \cite{Ish1974} for generalized nonexpansive mappings (see \cite{GFLFS2011}) and range to  involved examples like, e.g., algorithms approximating equilibrium points \cite{IY2009} or algorithms for solving problems in DC programming \cite{Mou2015}, and beyond.

\medskip

Even with this breadth of notions and schemes encompassed by the concept of Fej\'er monotonicity, there are many influential iteration schemes which, although related to the concept in principle, are not truly Fej\'er monotone and thus can not be covered by the previous study thereof. 

For one, there are iterations which are only partially Fej\'er in the sense that only a subsequence is properly Fej\'er monotone and the other elements connect to this subsequence just via weaker monotonicity requirements. Such a situation in particular arises in the context of ``alternating methods'', i.e.\ algorithms which alternate between two (or more) ``schemas'', e.g.\ like in recent modifications of the proximal point algorithm \cite{IH2019,MP2015} which incorporate inertia terms \emph{every other} term.

For another, iterations can be Fej\'er monotone w.r.t.\ functions measuring the distance of points other than metrics. The most prominent and instructive example for this case are Bregman distances as derived from Bregman's seminal work \cite{Bre1967} or even generalized Bregman distances in the sense of Burachik, Dao and Lindstrom \cite{BDL2021}. From that perspective, the theory of sequences which are Fej\'er monotone w.r.t.\ Bregman distances has been recognized, in a similar way as ordinary Fej\'er monotone sequences which are defined w.r.t.\ a norm or metric, as a powerful framework for unifying many convergence results in nonlinear analysis (see in particular the seminal work by Bauschke, Borwein and Combettes \cite{BBC2003}) and applications of this framework include many prominent iterative schemes in the context of Bregman distances (see in particular the references in \cite{BBC2003}).

Lastly, in particular in the recent years, methods have been studied where instead of solving subproblems by e.g.\ minimizing some sub-objective function relative to a fixed distance, these distances are allowed to vary along the iteration so that these subproblems can be tweaked adaptively. Particular examples for such iterations are the recent works on variable-metric forward-backward type methods \cite{BLPP2016,CPR2014,CB2014} and in particular similar methods that allow for varying Bregman distances \cite{BC2021,Ngu2017}, among many more. The treatment of all these methods relies on establishing some form of Fej\'er monotonicity where the distances in question are allowed to vary along the iteration and for the particular case of varying metrics induced by bilinear forms, an abstract treatment of such types of methods has been given in \cite{CB2013}.

\medskip

In this paper, we will study a generalized form of Fej\'er monotonicity which unifies all of these variants. Concretely, our particular emphasis will be on sequences where the property of Fej\'er monotonicity is relaxed in the following three ways:

\begin{enumerate}
\item We allow for a very general class of functions (in particular including metrics as well as large classes of Bregman distances) to measure distances between points in the (quasi-)Fej\'er monotonicity property.
\item We simultaneously allow that the sequence only fulfills a partial form of Fej\'er monotonicity where one subsequence retains a generalized form of (quasi-)Fej\'er monotonicity while the other sequence is only required to fulfill a very relaxed form of monotonicity which generalizes the conditions on the complementary subsequence currently in place in the literature on such methods.
\item We allow the general distances to vary along the (partial) Fej\'er monotonicity property, capturing the multiple notions of varying distances currently discussed in the literature.
\end{enumerate}

To that end, we give abstract and unified strong convergence results for such sequences over compact (or relatively compact and thus in particular finite-dimensional) spaces which are moreover quantitative in the sense that, depending only on few quantitative reformulations of the key properties of the sequence, they provide a construction for an effective and highly uniform rate of metastability in the sense of T. Tao \cite{Tao2008b,Tao2008a} and even a rate of convergence under an additional regularity assumption introduced here\footnote{The dependence on such an additional assumption is indeed necessary if one wants to obtain computable rates of convergence. This follows from adaptations (see \cite{Neu2015}) of a classical result in computable analysis by Specker \cite{Spe1949} which, in general, rule out effective rates of convergence for many common iteration schemes, even in the context where all parameters of the sequence like potential mappings, etc., are effective (see the discussion in \cite{Neu2015}).} that generalizes the notion of metric regularity introduced in \cite{KLAN2019}.\footnote{This notion of regularity introduced in \cite{KLAN2019} provides a unified approach to several regularity notions known from optimization such as error bounds and metric subregularity, among others (see the discussion in \cite{KLAN2019}).} These abstract and quantitative results arise as generalizations of previous constructions of rates of metastability and rates of convergence for Fej\'er and quasi-Fej\'er monotone sequences obtained in \cite{KLN2018,KLAN2019} (see also \cite{Pis2022}). Moreover, by ``reversing'' the process and forgetting about the quantitative aspects, we are then able to derive a new and very general ``non-quantitative'' convergence result for these generalized Fej\'er monotone sequences in the previous three senses which in particular generalizes the classical results on Fej\'er and quasi-Fej\'er monotone sequences as well as the general convergence results on Bregman monotone sequences obtained in \cite{BBC2003} and on variable metric Fej\'er monotone sequences obtained in \cite{CB2013} (at least in relatively compact spaces). Not only does this treatment therefore unify all of these different types of Fej\'er monotonicity but it also can accommodate new methods that would not have been covered by either of these various strands of notions of Fej\'er monotonicity individually (as also commented on further below).

\medskip

Similar to the circumstances of the quantitative results on Fej\'er monotone sequences from \cite{KLN2018,KLAN2019}, the present work has been obtained through the general methodological approach of proof mining, a program in mathematical logic conceptually originating from Kreisel's unwinding of proofs \cite{Kre1951,Kre1952} and brought to maturity by U. Kohlenbach and his collaborators, which aims to classify and extract the computational content of prima facie ``non-computational'' proofs (see \cite{Koh2008} for a comprehensive monograph on the subject, \cite{KO2003} for an early survey and \cite{Koh2019} for a survey of recent results up to 2019).

Proof mining has been extensively applied in the context of nonlinear (functional) analysis and optimization to derive quantitative results in numerous circumstances. Particularly relevant for the present paper are, e.g., the quantitative treatments of equilibrium problems \cite{PK2021}, zeros of differences of monotone operators \cite{Pis2022}, the proximal point algorithm in various settings \cite{Koh2020,Koh2021,Koh2021b,LS2018a,LS2018b} and the  composition of two firmly nonexpansive mappings in nonlinear spaces \cite{KLAN2017} as well as Korpelevich's extragradient method \cite{Pis2023} as those applications were obtained using the previously mentioned abstract results on Fej\'er monotone sequences \cite{KLN2018,KLAN2019}.

However, as common in the proof mining program, we want to emphasize that while this logic-based approach was crucial in the course of obtaining the present results, there is no explicit use of logical tools in the derivation of our results presented here and all relevant logical particularities are discussed, if at all, only in separate remarks.

\medskip

We expect that the abstract results established here will exhibit a similar range of applications in the future. To that end, the second half of the paper is devoted to two particular exemplary case studies. Here, we provide quantitative results on the previously mentioned alternating inertia proximal point algorithm \cite{MP2015} and its generalization given in \cite{IH2019} to averaged mappings (or, respectively, comonotone operators in the sense of \cite{BMW2020}) as well as on a modification of the well-known proximal point algorithm due to \cite{KKT2004} (see also the related work \cite{MT2004}) to monotone operators in Banach spaces (in the sense of Browder, see in particular \cite{Bro1968}) which also modifies the iteration to be of a style similar to that of Mann's iteration (see also \cite{Xu2002} for similar modifications in Hilbert spaces).

\medskip

As a further indication of the applicability of the results presented here, we however want to additionally mention that beyond the previously discussed applications, the methods seem to be immediately further applicable in the context of the recent works on proximal gradient algorithms \cite{IM2018}, on a modification of Korpelevich's extragradient algorithm for variational inequalities (see \cite{Kor1976}) to incorporate alternating inertia \cite{SI2020}, on extragradient methods for equilibrium problems \cite{SDLY2021}, on algorithms for the split feasibility problem \cite{LWW2022,SG2021} or on modifications of the Krasnosel'ski\u{i}-Mann iteration to incorporate alternating inertia \cite{DCHPR2022} as well as on algorithms using the Bregman distance in the proximal point style \cite{BuIu1997,Eck1993}, on Landweber's iteration in Banach spaces \cite{SLS2006} and potential extensions to Bregman distances, on iterations of Bregman retractions \cite{BC2003}, on iterations of Bregman projections \cite{AB1997}, on algorithms for solving the variational inequality problem in Banach spaces using the Bregman distance \cite{BI1998,BS2001} or on algorithms which are only quasi-Fej\'er w.r.t.\ Bregman distances like iterations of inexact orbits \cite{BRZ2006}. Further, also the previously mentioned works on variable-metric forward-backward type methods \cite{BLPP2016,CPR2014,CB2014} and on varying Bregman distances \cite{BC2021,Ngu2017} seem to immediately fit in the framework proposed here and could thus be quantitatively treated.

\medskip

Lastly, we want to note that the general approach developed here is crucially used in the recent work \cite{Pis2024a} to study a new notion of Bregman distance in metric spaces where, in that context, the generality of the results presented here was crucial for obtaining corresponding convergence results of sequences which are Fej\'er monotone w.r.t.\ this metric variant of the Bregman distance as it was neither covered by previous work on Bregman monotone iterations, being set in normed spaces, nor on Fej\'er monotone sequences in metric spaces, as this new Bregman distance in metric spaces is not a metric itself.

\section{Preliminaries: convergence theorems for (quasi-)Fej\'er monotone sequences}

Throughout, if not stated otherwise, we consider an underlying metric space $(X,d)$. In this section, we now recap the main results on the usual (quasi-)Fej\'er monotone sequences underlying this work. In that way, we mainly exposit the work \cite{KLN2018}, being closest in spirit to this paper. Therefore, we mostly follow the notational setup of \cite{KLN2018} if not stated otherwise. In particular, we write $\mathbb{R}_+$ for the interval $[0,\infty)$ and we write $\mathbb{R}^*_+$ for the interval $(0,\infty)$. Further, we use $\overline{B}\left(p,r\right)$ to denote the closed ball around $p$ with radius $r$ in $(X,d)$. Also, we write $n\dotdiv m=\max\{0,n-m\}$ for $n,m\in\mathbb{N}$.

\medskip

The central notion investigated by the authors in \cite{KLN2018} is the following form of Fej\'er monotonicity relative to the metric:

\begin{definition}[\cite{KLN2018}]
Given two functions $G,H:\mathbb{R}_+\to\mathbb{R}_+$ and a non-empty subset $F\subseteq X$, a sequence $(x_n)\subseteq X$ is called \emph{$(G,H)$-Fej\'er monotone w.r.t.\ $F$} if for all $n,m\in\mathbb{N}$ and all $p\in F$:
\[
H(d(p,x_{n+m}))\leq G(d(p,x_n)).
\]
\end{definition}

Following \cite{KLN2018}, the functions $G$ and $H$ are later assumed to satisfy the following continuity properties: we assume that $G$ satisfies that
\[
a_n\to 0\text{ implies } G(a_n)\to 0
\]
and that $H$ satisfies that
\[
H(a_n)\to 0\text{ implies } a_n\to 0,
\]
both for any $(a_n)\subseteq\mathbb{R}_+$. Those properties are equivalent, respectively, to the existence of corresponding moduli $\alpha_G,\beta_H:\mathbb{N}\to\mathbb{N}$ which satisfy
\[
a\leq\frac{1}{\alpha_G(k)+1}\rightarrow G(a)\leq\frac{1}{k+1}
\]
as well as
\[
H(a)\leq\frac{1}{\beta_H(k)+1}\rightarrow a\leq\frac{1}{k+1}
\]
for any $k\in\mathbb{N}$ and any $a\in\mathbb{R}_+$. Following \cite{KLN2018}, we call $\alpha_G$ a \emph{$G$-modulus for $G$} and $\beta_H$ an \emph{$H$-modulus for $H$}.

\medskip

For the set $F$, we assume (following \cite{KLN2018}) a representation via
\[
F=\bigcap_{k\in\mathbb{N}}AF_k
\]
where the sequence of the sets $AF_k\subseteq X$ is monotone decreasing, i.e.\ $AF_{k}\supseteq AF_{k+1}$. The elements of $AF_k$ are to be intuitively understood as being $k$-good approximations of points in $F$ but can take arbitrary forms is our general setting.

In that context, we will need to assume that the set $F$ is not only closed but closed in a special sense relative to the approximations $AF_k$:

\begin{definition}[\cite{KLN2018}]
The set $F$ is \emph{explicitly closed} w.r.t.\ $AF_k$ if for all $p\in X$:
\[
\forall N,M\in\mathbb{N}\left( AF_M\cap\overline{B}\left(p,\frac{1}{N+1}\right)\neq\emptyset\right)\rightarrow p\in F.
\]
\end{definition}

As is well known from the context of usual Fej\'er monotone sequences (see \cite{KLN2018}), obtaining a convergence result for such sequences requires an additional assumption in the form of an ``approximation property'', which here takes of the form of the following definition: 

\begin{definition}[\cite{KLN2018}]
We say that a sequence $(x_n)\subseteq X$ has \emph{approximate $F$-points \textnormal{(}\emph{w.r.t.} $(AF_k)$\textnormal{)}} if
\[
\forall k\in\mathbb{N}\exists N\in\mathbb{N}(x_N\in AF_k).
\]
\end{definition}

The main result on Fej\'er monotone sequences in compact spaces is now the following proposition (taken in this form from \cite{KLN2018}):

\begin{proposition}\label{pro:convResNormal}
Let $(X,d)$ be compact and $F\subseteq X$ be explicitly closed. If $(x_n)\subseteq X$ is $(G,H)$-Fej\'er monotone w.r.t.\ $F$ and has approximate $F$-points \textnormal{(}w.r.t.\ $(AF_k)$\textnormal{)}, then $(x_n)$ converges to some $x^*\in F$.
\end{proposition}

In many situations, Fej\'er monotonicity is considered in the context of an additional sequence of summable errors in the sense of the following definition, which is then called quasi-Fej\'er monotonicity:

\begin{definition}[\cite{KLN2018}]
Given two functions $G,H:\mathbb{R}_+\to\mathbb{R}_+$, a non-empty subset $F\subseteq X$ and a sequence $(\varepsilon_n)\subseteq\mathbb{R}_+$ with $\sum\varepsilon_n<\infty$, a sequence $(x_n)\subseteq X$ is called \emph{$(G,H)$-quasi-Fej\'er monotone w.r.t.\ $F$ and $(\varepsilon_n)$} if for all $n,m\in\mathbb{N}$ and all $p\in F$:
\[
H(d(p,x_{n+m}))\leq G(d(p,x_n))+\sum_{i=n}^{n+m-1}\varepsilon_i.
\]
\end{definition}

In the context of this extended notion, the assumption of the sequence having approximate $F$-points has to be upgraded to a $\liminf$-like property in the sense of the following definition:

\begin{definition}[\cite{KLN2018}]
We say that a sequence $(x_n)\subseteq X$ has the \emph{$\liminf$-property w.r.t.\ $F$ \textnormal{(}and $(AF_k)$\textnormal{)}} if
\[
\forall k,n\in\mathbb{N}\exists N\in\mathbb{N}(N\geq n\text{ and } x_N\in AF_k).
\]
\end{definition}

Then one obtains the following main convergence result for quasi-Fej\'er monotone sequences in compact spaces (again taken in this form from \cite{KLN2018}):

\begin{proposition}\label{pro:convRes}
Let $(X,d)$ be compact and $F\subseteq X$ be explicitly closed. If $(x_n)\subseteq X$ is $(G,H)$-quasi-Fej\'er monotone w.r.t.\ $F$ and $(\varepsilon_n)$ and has the $\liminf$-property w.r.t.\ $F$ \textnormal{(}and $(AF_k)$\textnormal{)}, then $(x_n)$ converges to some $x^*\in F$.
\end{proposition}

For these convergence results, the work \cite{KLN2018} then provided quantitative variants in the form of an explicit construction of a rate of metastability depending only on moduli witnessing uniform quantitative reformulations of the core components of the above result (as will be discussed in more detail later on). Further, the work \cite{KLAN2019} similarly provided explicit constructions even for full rates of convergence in the context of Proposition \ref{pro:convResNormal}, subject to an additional dependence on a modulus of regularity (as will also be discussed later on). This construction of rates of convergence for Fej\'er monotone sequences was extended to quasi-Fej\'er monotone sequences in \cite{Pis2022}.

\medskip

The main aim of this paper now is to generalize these quantitative results from \cite{KLN2018,KLAN2019} to much more general circumstances, uniformly covering all the Fej\'er-like sequences discussed in the introduction. While in \cite{KLN2018,KLAN2019} the main aim was at providing quantitative variants of already established convergence results, we here not only aim at quantitative convergence results for these generalized Fej\'er monotone sequences but, as discussed before, also use to these quantitative results to obtain new ``non-quantitative'' convergence results. The following section now formally introduces the desired generalizations.

\section{Generalized (quasi-)Fej\'er monotone sequences}

As discussed before, the general situation which we want to consider in this paper simultaneously treats the following three extensions of the usual metric notion of Fej\'er monotonicity: we want to allow

\begin{enumerate}
\item more general distance functions $\phi:X\times X\to\mathbb{R}_+$ to uniformly capture notions like metrics, Bregman distances and beyond,
\item a partial version of Fej\'er monotonicity where only a subsequence is properly Fej\'er monotone,
\item the distance function to be able to vary along the sequence in the Fej\'er-type property.
\end{enumerate}

The following subsections now discuss these different extensions. We will always define these notions already in their ``quasi-type'' variants, incorporating errors.

\subsection{(Weakly) triangular mappings}

The most general class of distance functions $\phi:X\times X\to\mathbb{R}_+$ which we want to permit is only characterized by a particular weak property which we will call \emph{weakly triangular} which derives as a uniform quantitative version of the property
\[
\phi(y,x)=\phi(y,z)=0\to \phi(x,z)=0,
\]
which is a particular special case of the triangle inequality for $\phi$. Concretely, we consider the following notion:
\begin{definition}
Let $a\in X$. A function $\phi:X\times X\to\mathbb{R}_+$ is called \emph{weakly triangular} (w.r.t.\ $a$) with modulus $\theta:\mathbb{N}^2\to\mathbb{N}$ if for any $k,b\in\mathbb{N}$ and any $x,y,z\in X$ such that $d(a,x),d(a,y),d(a,z)\leq b$, it holds that
\[
\phi(y,x),\phi(y,z)\leq\frac{1}{\theta(k,b)+1}\to \phi(x,z)\leq\frac{1}{k+1}.
\]
We call $\phi$ \emph{weakly triangular} if there exists an accompanying modulus and we say that a modulus is uniform if it does not depend on $b$.\footnote{In $B$-bounded metric spaces, such a uniform modulus can of course always be derived from a usual modulus $\theta(k,b)$ by considering $\theta'(k)=\theta(k,B)$.}
\end{definition}

We also consider a strengthening of this notion which connects a function $\phi$ with the underlying metric $d$ in a certain way.

\begin{definition}
Let $a\in X$. A function $\phi:X\times X\to\mathbb{R}_+$ is called \emph{triangular} (w.r.t.\ $a$) with modulus $\theta:\mathbb{N}^2\to\mathbb{N}$ if for any $k,b\in\mathbb{N}$ and any $x,y,z\in X$ such that $d(a,x),d(a,y),d(a,z)\leq b$, it holds that
\[
\phi(y,x),\phi(y,z)\leq\frac{1}{\theta(k,b)+1}\to d(x,z)\leq\frac{1}{k+1}.
\]
We call $\phi$ \emph{triangular} if there exists an accompanying modulus and we say that a modulus is uniform if it does not depend on $b$.\footnote{As before, such a uniform modulus can be derived from a usual modulus bounded metric spaces.}
\end{definition}

We also introduce a notion of comparison for two distance functions: for $\phi,\psi:X\times X\to \mathbb{R}_+$ and $A:\mathbb{N}\to\mathbb{N}$, we write $\phi\preceq^A \psi$ if
\[
\psi(x,y)\leq\frac{1}{A(k)+1}\to \phi(x,y)\leq\frac{1}{k+1}
\]
for any $x,y\in X$.

\medskip

As for a choice of naming, we want to note that being triangular can be ``conceived of'' as being stronger than weak triangularity as in many cases, it can be inferred that $\phi\preceq^A d$ for a suitable function $A$ (e.g.\ if $\phi$ is uniformly continuous) in which case a modulus for triangularity can be converted into a modulus of weak triangularity (see in particular the discussions in Section \ref{sec:ConsDist} later on).

\medskip

In the following, at the most basic level only this property of being (weakly) triangular will be required from $\phi$ and we in particular allow $\phi$ to not be symmetric and to not satisfy the full triangle inequality.

\begin{remark}\label{rem:wDist}
The notion of triangular distances is related to the notion of $w$-distances introduced by Kada, Suzuki and Takahashi in \cite{KST1996} (see also \cite{Tak2000}). Concretely, a map $\rho:X\times X\to\mathbb{R}_+$ is called a $w$-distance if
\begin{enumerate}
\item $\rho(x,z)\leq\rho(x,y)+\rho(y,z)$ for all $x,y,z\in X$,
\item $\rho(x,\cdot)$ is lower semicontinuous for any $x\in X$,
\item for any $\varepsilon>0$, there exists a $\delta>0$ such that for all $x,y,z\in X$:
\[
\rho(z,x),\rho(z,y)\leq\delta\to d(x,y)\leq\varepsilon.
\]
\end{enumerate}
Condition (3) is equivalent to $\rho$ being triangular. However being triangular is clearly more general than being a $w$-distance and, as mentioned before, we in the following in particular do not require the triangle inequality, i.e.\ condition (1) above, which allows us later to treat Bregman distances as these generally do not enjoy that property.
\end{remark}

Particular examples of such distances and their moduli will be discussed later on but we already want to give some easy instances here.

\begin{lemma}[essentially \cite{Tak2000}]\label{lem:exTriDist}
The following functions are weakly triangular and even triangular:\footnote{The examples are adapted from \cite{Tak2000} and as such many of these distances are $w$-distances, or can be turned into such with a few additional conditions, in the sense of the previous Remark \ref{rem:wDist}.}
\begin{enumerate}
\item Any metric is uniformly triangular and uniformly weakly triangular with the common modulus $\theta(k)=2k+1$.
\item Assume $X$ is a normed space with norm $\norm{\cdot}$. Then $\phi(x,y)=\norm{x}+\norm{y}$ and $\phi(x,y)=\norm{y}$ are both uniformly triangular and uniformly weakly triangular with the common modulus $\theta(k)=2k+1$.
\item Given $c>0$, define $\phi(x,y)=c$ for any $x,y\in X$. Then $\phi$ is uniformly triangular and uniformly weakly triangular with the common modulus $\theta(k)=\overline{c}+1$ where $\overline{c}\in\mathbb{N}$ satisfies $c\geq 1/(\overline{c}+1)$.
\item Let $T:X\to X$ be any mapping. Then $\phi(x,y)=\max\{d(Tx,y),d(Tx,Ty)\}$ is uniformly triangular and uniformly weakly triangular with the common modulus $\theta(k)=2k+1$.
\item Let $F$ be bounded and $c>0$ be a bound on its diameter. Define
\[
\phi(x,y)=\begin{cases}d(x,y)&\text{if }x,y\in F,\\c&\text{if }x\not\in F\text{ or }y\not\in F.\end{cases}
\]
Then $\phi$ is uniformly triangular and uniformly weakly triangular with the common modulus $\theta(k)=2n_0(k+1)\dotdiv 1$ where $n_0=\ceil*{1/c(k+1)}+1$.
\end{enumerate}
Further, the maximum and linear combination of finitely many (weakly) triangular mappings is again (weakly) triangular and moduli for the combinations can be computed from the moduli of the constituents. Lastly, if $f:X\to \mathbb{R}_+$ is any function and $\phi$ is a triangular mapping, then $\phi'(x,y)=\max\{f(x),\phi(x,y)\}$ is also triangular with the same modulus.
\end{lemma}

The proofs are immediate, so we omit them.

\medskip

Beyond these examples, we now just note two things: For one, note that the reference point $a$ used for ``bounding'' the elements $x,y,z$ is not crucial for the existence of a modulus since if $\theta$ is such a modulus of (weak) triangularity w.r.t.\ $a$ and $a'$ is such that $d(a,a')\leq r$, then $\theta'(k,b)=\theta(k,b+r)$ is a modulus of (weak) triangularity w.r.t.\ $a'$.

Also, we note for another that if $\phi$ is continuous in both arguments, then a simple compactness argument shows that over compact spaces $X$ the property
\[
\phi(y,x)=\phi(y,z)=0\to \phi(x,z)=0
\]
already implies that $\phi$ is uniformly weakly triangular and that the property 
\[
\phi(y,x)=\phi(y,z)=0\to d(x,z)=0
\]
already implies that $\phi$ is uniformly triangular.

\begin{remark}[For logicians]
The underlying logical methods from proof mining (see in particular \cite{GeK2008,Koh2005} and the treatment in \cite{Koh2008}) actually guarantee that corresponding moduli of (weak) triangularity can be extracted from a large class of proofs of the above respective properties
\[
\phi(y,x)=\phi(y,z)=0\to \phi(x,z)=0
\]
and
\[
\phi(y,x)=\phi(y,z)=0\to d(x,z)=0,
\]
even potentially in the absence of any compactness assumptions.
\end{remark} 

\medskip

Relating to such distance functions, we will in the following consider an associated notion of Fej\'er monotonicity (derived by generalizing the previous definition):

\begin{definition}
Let $\phi:X\times X\to\mathbb{R}_+$ be any function. Given two functions $G,H:\mathbb{R}_+\to\mathbb{R}_+$, a non-empty subset $F\subseteq X$ and a sequence $(\varepsilon_n)\subseteq\mathbb{R}_+$ with $\sum\varepsilon_n<\infty$, a sequence $(x_n)\subseteq X$ is called \emph{$\phi$-$(G,H)$-quasi-Fej\'er monotone w.r.t.\ $F$ and $(\varepsilon_n)$} if for all $n,m\in\mathbb{N}$ and all $p\in F$:
\[
H(\phi(p,x_{n+m}))\leq G(\phi(p,x_n))+\sum_{i=n}^{n+m-1}\varepsilon_i.
\]
\end{definition}

\subsection{Partial Fej\'er monotonicity}

Now, regarding the ``partial Fej\'er monotonicity'', we consider the following more concrete version of the previously only loosely defined notion. By partial Fej\'er monotonicity, we want to understand that only a subsequence is properly Fej\'er monotone while the other elements ``connect''  to this subsequence just by some very weak form of monotonicity which is some form of ``slowed-down'' Fej\'er monotonicity.

\medskip

Concretely, we will (in some sense without loss of generality) identify this designated subsequence with the elements of the sequence at an even index and thus require that for a given sequence $(x_n)$, only $(x_{2n})$ is Fej\'er monotone w.r.t.\ $F$. For the odd elements, we require only a weak \emph{$f$-monotonicity} property that relates the odd elements $x_{2(n+m)+1}$ to previous even elements $x_{2f(n)}$ where $f:\mathbb{N}\to\mathbb{N}$ is some number-theoretic function which is nondecreasing but otherwise can be mostly arbitrary (and in particular can be essentially arbitrarily slow increasing). Additionally incorporating the previous generalizations introduced in the context of Fej\'er monotonicity like the functions $G,H$ as well as using a general function $\phi$ for measuring distances, we arrive at the following definition:

\begin{definition}
Let $\phi:X\times X\to\mathbb{R}_+$ be any function. Given two functions $G,H:\mathbb{R}_+\to\mathbb{R}_+$, a non-empty subset $F\subseteq X$, a sequence $(\varepsilon_n)\subseteq\mathbb{R}_+$ with $\sum\varepsilon_n<\infty$ and given a function $f:\mathbb{N}\to\mathbb{N}$, a sequence $(x_n)\subseteq X$ is called \emph{$\phi$-$(G,H)$-quasi-$f$-monotone w.r.t.\ $F$ and $(\varepsilon_n)$} if for all $n,m\in\mathbb{N}$ and all $p\in F$:
\[
H(\phi(p,x_{2(n+m)+1}))\leq G(\phi(p,x_{2f(n)}))+\sum_{i=f(n)}^{n+m}\varepsilon_i.
\]
\end{definition}

\subsection{Variable distances}

At last, to capture the various methods employing changing metrics or Bregman distances, we want to consider the above notions relativized to a sequence of distance functions $(\phi_n)$. 

\begin{definition}
Let $(\phi_n)$ be a sequence of functions $\phi_n:X\times X\to\mathbb{R}_+$. Given two functions $G,H:\mathbb{R}_+\to\mathbb{R}_+$, a non-empty subset $F\subseteq X$ and a sequence $(\varepsilon_n)\subseteq\mathbb{R}_+$ with $\sum\varepsilon_n<\infty$, a sequence $(x_n)\subseteq X$ is called \emph{$(\phi_n)$-$(G,H)$-quasi-Fej\'er monotone w.r.t.\ $F$ and $(\varepsilon_n)$} if for all $n,m\in\mathbb{N}$ and all $p\in F$:
\[
H(\phi_{n+m}(p,x_{n+m}))\leq G(\phi_n(p,x_n))+\sum_{i=n}^{n+m-1}\varepsilon_i.
\]
\end{definition}

Similarly, we also consider a version of the accompanying notion of $f$-monotonicity, relativized to a sequence $(\phi_n)$:

\begin{definition}
Let $(\phi_n)$ be a sequence of functions $\phi_n:X\times X\to\mathbb{R}_+$. Given two functions $G,H:\mathbb{R}_+\to\mathbb{R}_+$, a non-empty subset $F\subseteq X$, a sequence $(\varepsilon_n)\subseteq\mathbb{R}_+$ with $\sum\varepsilon_n<\infty$ and given a function $f:\mathbb{N}\to\mathbb{N}$, a sequence $(x_n)\subseteq X$ is called \emph{$(\phi_n)$-$(G,H)$-quasi-$f$-monotone w.r.t.\ $F$ and $(\varepsilon_n)$} if for all $n,m\in\mathbb{N}$ and all $p\in F$:
\[
H(\phi_{2(n+m)+1}(p,x_{2(n+m)+1}))\leq G(\phi_{2f(n)}(p,x_{2f(n)}))+\sum_{i=f(n)}^{n+m}\varepsilon_i.
\]
\end{definition}

Concretely, this constitutes our final notion of generalized Fej\'er monotonicity:

\begin{definition}
Let $(\phi_n)$ be a sequence of functions $\phi_n:X\times X\to\mathbb{R}_+$. Given two functions $G,H:\mathbb{R}_+\to\mathbb{R}_+$, a non-empty subset $F\subseteq X$, a sequence $(\varepsilon_n)\subseteq\mathbb{R}_+$ with $\sum\varepsilon_n<\infty$ and given a function $f:\mathbb{N}\to\mathbb{N}$, we call a sequence $(x_n)\subseteq X$ \emph{partially $(\phi_n)$-$(G,H)$-quasi-Fej\'er monotone w.r.t.\ $F$, $f$ and $(\varepsilon_n)$} if $(x_{2n})$ is $(\phi_{2n})$-$(G,H)$-quasi-Fej\'er monotone w.r.t.\ $F$ and $(\varepsilon_{n})$ and $(x_n)$ is $(\phi_n)$-$(G,H)$-quasi-$f$-monotone w.r.t.\ $F$ and $(\varepsilon_n)$.
\end{definition}

The exemplary applications given later, in particular also in combination with the work \cite{Pis2024a}, then illustrate the versatility and applicability of this generalized Fej\'er monotonicity.
 
\section{Rates of metastability for generalized Fej\'er monotone sequences}

One of the main results of the first part of this paper will be theorems that extend the usual convergence result for quasi-Fej\'er monotone sequences presented in Proposition \ref{pro:convRes} to these generalized Fej\'er monotone sequences. This generalization will be obtained through a preceding quantitative result by ``forgetting'' the quantitative information and moving to a simple convergence conclusion.

\medskip

These quantitative versions, which we therefore shall establish first, form the other main cluster of results of this first part. They themselves are obtained by generalizing the quantitative versions of Proposition \ref{pro:convRes} due to \cite{KLN2018,KLAN2019}. To that end, we generalize the moduli introduced in \cite{KLN2018,KLAN2019} that witness uniform quantitative versions of the central assumptions of Proposition \ref{pro:convRes} (which are equivalent to the original assumptions in the context of compact spaces).

\begin{remark}[For logicians]
Similar to \cite{KLN2018,KLAN2019}, these moduli are introduced as guided by the underlying proof-theoretic methodology of proof mining. In that way, these moduli all have the essential property that the so-called logical metatheorems used in the underlying proof mining program (recall the previously mentioned references) guarantee the existence of a quantitative variant of Proposition \ref{pro:convRes} similar to the ones presented in \cite{KLN2018,KLAN2019} which depends only on such moduli. Moreover, in many cases, the moduli in question can be extracted from proofs of the underlying non-quantitative properties at hand, again as guided by the same logical methodology.
\end{remark}

In our case, the main generalizations will be obtained by extending the previous moduli from \cite{KLN2018} to be applicable in the context of a general mapping $\phi$. Concretely, the main assumptions of Proposition \ref{pro:convRes} are translated to the following $\phi$-relative variants and associated uniform quantitative versions in similarity to \cite{KLN2018} which will prominently feature in the quantitative results later on:

\begin{center}
\begin{tabular}{c||c}
total $\phi$-boundedness & moduli of total $\phi$-boundedness $\gamma$\\
\hline
explicit $\phi$-closedness & moduli of uniform $\phi$-closedness $\omega,\delta$\\
\hline
approximate $F$-points & approximate $F$-point bounds $\Phi$\\
\hline
$\liminf$-property w.r.t.\ $F$ & $\liminf$-bounds $\Phi$ w.r.t.\ $F$\\
\hline
$(\phi_n)$-$(G,H)$-quasi-Fej\'er monotonicity & moduli of uniform $(\phi_n)$-$(G,H)$-quasi-Fej\'er monotonicity $\chi$\\
\hline
$(\phi_n)$-$(G,H)$-quasi-$f$-monotonicity & moduli of uniform $(\phi_n)$-$(G,H)$-quasi-$f$-monotonicity $\zeta$
\end{tabular}
\end{center}

The main new quantitative notion considered here is the last modulus of uniform $(\phi_n)$-$(G,H)$-quasi-$f$-monotonicity whose details and motivation we will discuss later in a more precise way. The other notions, while here considered in variants relative to the distances $(\phi_n)$, arise as straightforward modifications of the metric-based notions introduced in \cite{KLN2018}. We nevertheless, both due to the alteration by the functions $(\phi_n)$ as well as for reasons of self-containedness, discuss also those moduli briefly.\\

While these moduli and the corresponding convergence results proved later will be presented in an abstract fashion, we refer to \cite{KLN2018}, the applications derived thereof as well as to the applications presented later on in this paper and those developed in \cite{Pis2024a} for many versatile examples of concrete instantiations of these moduli.

\subsection{Moduli of total boundedness}

Compactness features as an essential assumption already in the strong convergence results formulated in Propositions \ref{pro:convResNormal} and \ref{pro:convRes}. As such, also in this paper, usual compactness as well as a version relativized to general distances will be (at least at first) feature crucially in the context of the upcoming strong convergence results. The formulation of compactness that is most crucial here is that obtained via the notion of total boundedness: a set $A\subseteq X$ is called \emph{totally bounded} if there exists an $\varepsilon$-net covering it for any $\varepsilon>0$. An alternative characterization of this notion commonly used in the context of proof mining was introduced in \cite{Ger2008}:
\begin{definition}[\cite{Ger2008}]
For $A\neq\emptyset$, $\gamma:\mathbb{N}\to\mathbb{N}$ is called a \emph{modulus of total boundedness for $A$} if for all $k\in\mathbb{N}$ and any $(x_n)\subseteq A$:
\[
\exists 0\leq i<j\leq\gamma(k)\left( d(x_i,x_j)\leq\frac{1}{k+1}\right).
\]
\end{definition}
Indeed, $A$ is totally bounded if, and only if $A$ has a modulus of total boundedness (see \cite{KLN2018}, Proposition 2.4).\footnote{Notice that the notion of a modulus of total boundedness from \cite{Ger2008} is discussed in \cite{KLN2018} under the name of a II-modulus of total boundedness.}

\medskip

As mentioned before, we will require this notion in the context of general functions $\phi:X\times X\to\mathbb{R}_+$.

\begin{definition}
We say that a set $A\subseteq X$ is totally $\phi$-bounded if for any $k\in\mathbb{N}$ there are $a_0,\dots,a_{j_k}\in X$ such that
\[
\forall x\in A\exists 0\leq i\leq j_k\left(\phi(a_i,x)\leq\frac{1}{k+1}\right).
\]
\end{definition}

This gives rise to the following analogous notion of modulus:

\begin{definition}
Let $\phi:X\times X\to\mathbb{R}_+$ be any function. For $A\neq\emptyset$, $\gamma:\mathbb{N}\to\mathbb{N}$ is called a \emph{modulus of total $\phi$-boundedness for $A$} if for all $k\in\mathbb{N}$ and any $(x_n)\subseteq A$:
\[
\exists 0\leq i<j\leq\gamma(k)\left( \phi(x_j,x_i)\leq\frac{1}{k+1}\right).
\]
\end{definition}

The existence of such a modulus is in fact equivalent to being totally $\phi$-bounded, at least if $\phi$ is weakly triangular, as the following proposition shows (which is constructed in analogy to Proposition 2.4 from \cite{KLN2018}):

\begin{proposition}
Let $\phi$ be uniformly weakly triangular. If a set $A\subseteq X$ is totally $\phi$-bounded, then it has a modulus of total $\phi$-boundedness. Conversely, if $\phi$ additionally satisfies $\phi(x,x)=0$ for all $x\in X$, then a set $A\subseteq X$ that has a modulus of total $\phi$-boundedness is also totally $\phi$-bounded.
\end{proposition}
\begin{proof}
Throughout, let $\theta$ be a modulus of uniform weak triangularity of $\phi$.

\medskip

Note that a set $A$ is clearly totally $\phi$-bounded if, and only if, there exists a modulus $\alpha:\mathbb{N}\to\mathbb{N}$ (derived from the notion of I-modulus from \cite{KLN2018}) satisfying that for any $k\in\mathbb{N}$ there are $a_0,\dots,a_{\alpha(k)}\in X$ such that
\[
\forall x\in A\exists 0\leq i\leq \alpha(k)\left(\phi(a_i,x)\leq\frac{1}{k+1}\right).
\]
We show that such a modulus $\alpha$ and a modulus of total $\phi$-boundedness can be translated into each other.

\medskip

Suppose now first that $\alpha$ is a modulus for a set $A$ as above and define $\gamma(k)=\alpha(\theta(k))+1$. We show that $\gamma$ is a modulus of total $\phi$-boundedness of $A$. For this, let
\[
a_0,\dots,a_{\alpha(\theta(k))}\in X
\]
be such that
\[
\forall x\in A\exists 0\leq i\leq\alpha(\theta(k))\left( \phi(a_i,x)\leq\frac{1}{\theta(k)+1}\right).
\]
Let $x_0,\dots,x_{\alpha(\theta(k))+1}$ be any sequence from $A$. By the pigeonhole principle, there are $i<j$ such that for some $l\leq \alpha(\theta(k))$:
\[
\phi(a_l,x_i),\phi(a_l,x_j)\leq\frac{1}{\theta(k)+1}.
\]
Thus we in particular have
\[
\phi(x_j,x_i)\leq\frac{1}{k+1}
\]
as $\theta$ is a modulus of uniform triangularity of $\phi$.

\medskip

Conversely, suppose that $\gamma$ is a modulus of total $\phi$-boundedness for a set $A$ and define $\alpha(k)=\gamma(\theta(k))-1$. Note first that, by the defining property of $\gamma$, we have $\gamma(k)\geq 1$ for any $k$ as $\gamma$ is a modulus so that $\alpha$ is well-defined. We show that $\alpha$ witnesses that $A$ is totally $\phi$-bounded as above. Suppose not, i.e.\ there exists a $k\in\mathbb{N}$ such that
\[
\forall a_0,\dots,a_{\gamma(\theta(k))-1}\in X\exists x\in A\forall 0\leq i\leq \gamma(\theta(k))-1\left( \phi(a_i,x)>\frac{1}{k+1}\right).\tag{$*$}
\]
By induction we show that for all $l\leq\gamma(\theta(k))$:
\[
\exists \beta_0,\dots,\beta_{l}\in A\forall 0\leq i<j\leq l\left(\phi(\beta_j,\beta_i)>\frac{1}{\theta(k)+1}\right)
\]
which, for $l=\gamma(\theta(k))$, in particular contradicts that $\gamma$ is a modulus of total $\phi$-boundedness. For $l=0$, pick $\beta_0\in A$ arbitrary (where we assume w.l.o.g.\ $A\neq\emptyset$). For the step from $l$ to $l+1$, let $\beta_0,\dots,\beta_l$ be chosen from the induction hypothesis. Now apply ($*$) to 
\[
a_i=\begin{cases}\beta_i&\text{if }i\leq l,\\\beta_l&\text{if }l\leq i\leq\gamma(\theta(k))-1,\end{cases}
\]
to get an $x\in A$ such that
\[
\phi(a_i,x)>\frac{1}{k+1}\text{ for all }i\leq l,
\]
i.e.\ in particular $\phi(\beta_i,x)>1/(k+1)$ for all $i\leq l$. Now suppose that
\[
\phi(x,\beta_i)\leq\frac{1}{\theta(k)+1}.
\]
Then as $\phi(x,x)=0$, we get
\[
\phi(\beta_i,x)\leq\frac{1}{k+1}
\]
which is a contradiction to before. Thus we have $\phi(x,\beta_i)>\frac{1}{\theta(k)+1}$ and thus $\beta_0,\dots,\beta_l,\beta_{l+1}$ with $\beta_{l+1}=x$ satisfies the desired claim for $l+1$.
\end{proof}

\subsection{Moduli of uniform closedness}

In \cite{KLN2018}, the explicit closedness of $F$ w.r.t.\ $AF_k$ is upgraded to the existence of moduli of uniform closedness in the following sense:

\begin{definition}[\cite{KLN2018}]
The set $F$ is called \emph{uniformly closed with moduli $\delta_F,\omega_F:\mathbb{N}\to\mathbb{N}$ \textnormal{(}w.r.t.\ $(AF_k)$\textnormal{)}} if for all $k\in\mathbb{N}$ and all $p,q\in X$:
\[
q\in AF_{\delta_F(k)}\text{ and }d(q,p)\leq\frac{1}{\omega_F(k)+1}\text{ implies }p\in AF_k.
\]
\end{definition}

Following \cite{KLN2018}, this quantitative reformulation is motivated by rewriting the definition of explicit closedness via 
\begin{align*}
&\forall p\in X\left(\forall N,M\in\mathbb{N}\left( AF_M\cap\overline{B}\left(p,\frac{1}{N+1}\right)\neq\emptyset\right)\to p\in F\right)\\
&\qquad\qquad\qquad\equiv \forall p\in X\forall k\in\mathbb{N}\left(\forall N,M\in\mathbb{N}\left( AF_M\cap\overline{B}\left(p,\frac{1}{N+1}\right)\neq\emptyset\right)\to p\in AF_k\right)\\
&\qquad\qquad\qquad\equiv \forall p\in X\forall k\in\mathbb{N}\exists N,M\in\mathbb{N}\forall q\in X\left( q\in AF_M\text{ and }d(q,p)\leq\frac{1}{N+1}\to p\in F\right)
\end{align*}
and then requiring that corresponding moduli exist which bound (and thus witness) the quantifiers over $N,M$ uniformly in $k$, i.e.\ without dependence on $p$.

\medskip

As with moduli of total $\phi$-boundedness, we also require this notion to be relativized to general functions $\phi:X\times X\to\mathbb{R}_+$.

\begin{definition}
Let $\phi:X\times X\to\mathbb{R}_+$ be any function. The set $F$ is \emph{explicitly $\phi$-closed} w.r.t.\ $AF_k$ if for all $p\in X$:
\[
\forall N,M\in\mathbb{N}\left( AF_M\cap\overline{B}_\phi\left(p,\frac{1}{N+1}\right)\neq\emptyset\right)\rightarrow p\in F,
\]
where we write $\overline{B}_\phi(p,r)=\{x\in X\mid \phi(x,p)\leq r\}$.
\end{definition}

In a similar vein as the moduli of uniform closedness arose above from a suitable rewriting of the property of explicit closedness, we now define moduli of uniform $\phi$-closedness:

\begin{definition}
Let $\phi:X\times X\to\mathbb{R}_+$ be any function. The set $F$ is called \emph{uniformly $\phi$-closed with moduli $\delta_F,\omega_F:\mathbb{N}\to\mathbb{N}$ \textnormal{(}w.r.t.\ $(AF_k)$\textnormal{)}} if for all $k\in\mathbb{N}$ and all $p,q\in X$:
\[
q\in AF_{\delta_F(k)}\text{ and }\phi(q,p)\leq\frac{1}{\omega_F(k)+1}\text{ implies }p\in AF_k.
\]
\end{definition}

\begin{remark}[For logicians]
Similar to before, also these moduli of uniform closedness can be extracted by separate applications of the logical metatheorems of proof mining to the respective proofs of the explicit closedness property (see also the discussion in \cite{KLN2018}). 
\end{remark}

\subsection{Approximate $F$-point bounds}

The approximate $F$-point property will be quantitatively witnessed by an accompanying bound on the index in the sense of the following definition:

\begin{definition}[\cite{KLN2018}]
A bound $\Phi(k)$ on the quantifier ``$\exists N\in\mathbb{N}$'' in the definition of approximate $F$-points detailed before, i.e.\ a bound such that
\[
\forall k\in\mathbb{N}\exists N\leq\Phi(k)\left( x_{N}\in AF_k\right),
\]
is called an \emph{approximate $F$-point bound for $(x_n)$ \textnormal{(}w.r.t.\ $(AF_k)$\textnormal{)}}.
\end{definition}

Throughout, we will always assume that any approximate $F$-point bound $\Phi$ is monotone nondecreasing (note that this is without loss of generality as given an arbitrary approximate $F$-point bound $\Phi$, the function $\Phi^M(k):=\max\{\Phi(j)\mid j\leq k\}$ is also an approximate $F$-point bound and monotone nondecreasing, see also \cite{KLN2018}).

\begin{remark}[For logicians]
The extractability of these bounds from proofs of the underlying property of approximate $F$-points via the logical metatheorems of proof mining can in particular be guaranteed if the relation $x_N\in AF_k$ can be expressed by an existential formula in the language used by these metatheorems (see also in particular the discussion in \cite{KLN2018}).
\end{remark}

\subsection{$\liminf$-bounds w.r.t.\ $F$}

The $\liminf$-property that extends the approximate $F$-points as an assumption in the context of quasi-Fej\'er monotone sequences will be quantitatively witnessed by a so-called $\liminf$-bound in the sense of \cite{KLN2018}:
\begin{definition}[\cite{KLN2018}]
A bound $\Phi(k,n)$ is a \emph{$\liminf$-bound for $(x_n)$ w.r.t.\ $F$ \textnormal{(}and $(AF_k)$\textnormal{)}} if
\[
\forall k,n\in\mathbb{N}\exists N\in [n;\Phi(k,n)](x_{N}\in AF_k).
\]
\end{definition}

Also here, we will throughout assume that any $\liminf$-bound $\Phi$ is monotone nondecreasing in both arguments (which is, by a similar argument as with the approximate $F$-point bounds, again without loss of generality).

\begin{remark}[For logicians]
As before with the approximate $F$-point bounds, the extractability of such bounds from proofs of the $\liminf$-property can in particular be guaranteed if $x_N\in AF_k$ can be expressed by an existential formula.
\end{remark}

\subsection{Moduli of uniform Fej\'er monotonicity}

In the case of Fej\'er monotone sequences, we move to the notion of uniform Fej\'er monotonicity in the sense of \cite{KLN2018} with accompanying modulus:
\begin{definition}[\cite{KLN2018}]
A sequence $(x_n)\subseteq X$ is called \emph{uniformly $(G,H)$-Fej\'er monotone w.r.t.\ $F$ \textnormal{(}and $(AF_k)$\textnormal{)}} if for all $r,n,m\in\mathbb{N}$, there exists a $k\in\mathbb{N}$ such that for all $p\in AF_k$:
\[
\forall l\leq m\left( H(d(p,x_{n+l}))< G(d(p,x_n))+\frac{1}{r+1}\right).
\]
A modulus of uniform $(G,H)$-Fej\'er monotonicity w.r.t.\ $F$ (and $(AF_k)$) is a function $\chi(n,m,r)$ bounding (and thus witnessing) such a $k$ in terms of the $r,n,m$.
\end{definition}

Now, for our generalized notion of Fej\'er monotonicity involving sequences of distance functions $(\phi_n)$, we consider the following natural extension of this notion of modulus (also incorporating  the error sequence featuring in  quasi-Fej\'er monotone sequences, as already discussed in \cite{KLN2018}):

\begin{definition}
Let $(\phi_n)$ be a sequence of functions $\phi_n:X\times X\to\mathbb{R}_+$. A sequence $(x_n)\subseteq X$ is called \emph{uniformly $(\phi_n)$-$(G,H)$-quasi-Fej\'er monotone w.r.t.\ $F$ \textnormal{(}and $(AF_k)$\textnormal{)} with errors $(\varepsilon_n)$} if for all $r,n,m\in\mathbb{N}$, there exists a $k\in\mathbb{N}$ such that for all $p\in AF_k$:
\[
\forall l\leq m\left( H(\phi_{n+l}(p,x_{n+l}))< G(\phi_n(p,x_n))+\sum_{i=n}^{n+l-1}\varepsilon_i+\frac{1}{r+1}\right).
\]
A modulus of uniform $(\phi_n)$-$(G,H)$-quasi-Fej\'er monotonicity w.r.t.\ $F$ (and $(AF_k)$) is a function $\chi(n,m,r)$ bounding (and thus witnessing) such a $k$ in terms of the $r,n,m$.
\end{definition}

If all $\phi_n$ coincide with one function $\phi$, then we call such a modulus just a \emph{modulus of uniform $\phi$-$(G,H)$-quasi-Fej\'er monotonicity} and we omit the $\phi$ if it coincides with the metric $d$.

\medskip

A simple compactness argument shows that if $F$ is explicitly closed w.r.t.\ $AF_k$, all $\phi_n$ are continuous in their left argument and $G,H$ are continuous, then normal $(\phi_n)$-$(G.H)$-quasi-Fej\'er monotonicity already implies uniform $(\phi_n)$-$(G.H)$-quasi-Fej\'er monotonicity (i.e.\ the existence of such a modulus).

\begin{remark}[For logicians]
For a discussion on the extractability of such moduli from proofs of the underlying properties of the various forms of Fej\'er monotonicity (which in essence can be guaranteed if the property $p\in AF_k$ can be expressed by a universal formula in the underlying language of the metatheorems), we again refer to the discussion given in \cite[Section 4.1]{KLN2018}.
\end{remark}

\subsection{Moduli of uniform $f$-monotonicity}

We now want to discuss the new quantitative notion necessary for the following results on rates of metastability for generalized Fej\'er monotone sequences. To this end, we want to give a more detailed account on the reasoning leading up to this notion (which is similar to that presented in \cite{KLN2018} for the moduli of uniform Fej\'er monotonicity): Consider the $(\phi_n)$-$(G,H)$-quasi-$f$-monotonicity property, i.e.
\[
H(\phi_{2(n+m)+1}(p,x_{2(n+m)+1}))\leq G(\phi_{2f(n)}(p,x_{2f(n)}))+\sum_{i=f(n)}^{n+m}\varepsilon_i
\]
for all $n,m\in\mathbb{N}$ and all $p\in F$. This property can be rewritten into
\begin{align*}
&\forall n,m\in\mathbb{N}\forall p\in X\Bigg( \forall k\in\mathbb{N}(p\in AF_k)\\
&\qquad\rightarrow \forall r\in\mathbb{N}\forall l\leq m\Bigg(H(\phi_{2(n+l)+1}(p,x_{2(n+l)+1}))<G(\phi_{2f(n)}(p,x_{2f(n)}))+\sum_{i=f(n)}^{n+l}\varepsilon_i+\frac{1}{r+1}\Bigg)\Bigg)
\end{align*}
which in turn is equivalent to
\begin{align*}
&\forall r,n,m\in\mathbb{N}\forall p\in X\exists k\in\mathbb{N}\Bigg( p\in AF_k\\
&\qquad\rightarrow\forall l\leq m\Bigg(H(\phi_{2(n+l)+1}(p,x_{2(n+l)+1}))<G(\phi_{2f(n)}(p,x_{2f(n)}))+\sum_{i=f(n)}^{n+l}\varepsilon_i+\frac{1}{r+1}\Bigg)\Bigg).
\end{align*}

Like with uniform Fej\'er monotonicity, we now call sequence uniformly $f$-monotone if we can bound (and thus witness) such a $k$ uniformly in $p$, i.e.\ only depending on $r,n,m\in\mathbb{N}$. This lead us to the following definition:

\begin{definition}
Let $(\phi_n)$ be a sequence of functions $\phi_n:X\times X\to\mathbb{R}_+$. A sequence $(x_n)$ is called \emph{uniformly $(\phi_n)$-$(G,H)$-quasi-$f$-monotone w.r.t.\ $F$ \textnormal{(}and $(AF_k)$\textnormal{)} with errors $(\varepsilon_n)$} if for all $r,n,m\in\mathbb{N}$, there exists a $k\in\mathbb{N}$ such that for all $p\in AF_k$:
\[
\forall l\leq m\left(H(\phi_{2(n+l)+1}(p,x_{2(n+l)+1}))<G(\phi_{2f(n)}(p,x_{2f(n)}))+\sum_{i=f(n)}^{n+l}\varepsilon_i+\frac{1}{r+1}\right).
\]
A modulus of \emph{uniform $(\phi_n)$-$(G,H)$-quasi-$f$-monotonicity w.r.t.\ $F$ \textnormal{(}and $(AF_k)$\textnormal{)} with errors $(\varepsilon_n)$} is an upper bound $\zeta(n,m,r)$ (and consequently a realizer) for $k$ in terms of $r,n,m$.
\end{definition}

Again, a simple compactness argument shows that the notions of normal and uniform $(\phi_n)$-$(G,H)$-quasi-$f$-monotonicity coincide in compact spaces if $F$ is explicitly closed and all $\phi_n$ as well as $G,H$ are suitably continuous.

\medskip

If all $\phi_n$ coincide with one function $\phi$, then we call such a modulus just a \emph{modulus of uniform $\phi$-$(G,H)$-quasi-$f$-monotonicity} and we omit the $\phi$ as before if it coincides with the metric $d$.

\begin{remark}[For logicians]
As the above motivating discussion shows, similar to the case of the modulus of uniform Fej\'er monotonicity (see again the discussion in \cite{KLN2018}), such a $\zeta$ can be extracted from a proof of the corresponding property, provided that $p\in AF_k$ can be written in a purely universal way.
\end{remark}

\subsection{A quantitative convergence result}

The previously detailed moduli can now be combined to a first quantitative convergence result, generalizing the quantitative versions of the convergence of (quasi-)Fej\'er monotone sequences from Propositions \ref{pro:convResNormal} and Proposition \ref{pro:convRes}, that is Theorems 5.1, 5.3 and 6.4 of \cite{KLN2018}. The proof of the following theorem is, in that way, an extension of the proofs for these theorems presented in \cite{KLN2018}.

\medskip

For this, we just need to discuss some last minor assumptions and their quantitative versions. The first such assumption is the summability of the error terms which is quantitatively witnessed via a Cauchy rate $\xi$:

\begin{definition}
A function $\xi:\mathbb{N}\to\mathbb{N}$ is a Cauchy rate for $\sum\varepsilon_n<\infty$ if 
\[
\forall k\in\mathbb{N}\left(\sum_{i=\xi(k)}^\infty\varepsilon_i\leq\frac{1}{k+1}\right).
\]
\end{definition}

Further, inspired by the conditions on the varying distances exhibited in e.g.\ \cite{BC2021,CB2013,CB2014,Ngu2017}, we will require that the sequence of distance functions $(\phi_n)$ behaves suitably regular in the sense that the functions converge from below and are bounded from below in a certain sense. For the former, we will concretely assume that the sequence $(\phi_n)$ approximates another function $\phi$ from below with a rate $\pi$ in the following sense:

\begin{definition}
A function $\pi:\mathbb{N}\to\mathbb{N}$ is a rate of convergence for $\phi_n\nearrow\phi$ uniformly if
\[
\forall k\in\mathbb{N}\forall n\geq \pi(k)\forall x,y\in X\left(\phi_n(x,y)\leq \phi(x,y)+\frac{1}{k+1}\right).
\]
\end{definition}

For the latter, using the previous notion of comparison $\preceq^A$ of two distance functions via a given function $A$, we will assume that there exist functions $\psi:X\times X\to\mathbb{R}_+$ and $A:\mathbb{N}\to\mathbb{N}$ such that $\psi\preceq^A\phi_n$ for all $n$ (together with some further properties of $\psi$). This in particular encompasses the common assumption from the literature that $\alpha\psi(x,y)\leq \phi_n(x,y)$ for all $x,y\in X$ and some $\alpha>0$, as in that case, if $\phi_n(x,y)\leq 1/(A(k)+1)$ for 
\[
A(k)=(a+1)(k+1)\dotdiv 1
\]
with $\alpha\geq 1/(a+1)$ for some $a\in\mathbb{N}$, then
\[
\frac{1}{a+1}\psi(x,y)\leq \alpha\psi(x,y)\leq \phi_n(x,y)\leq\frac{1}{A(k)+1}\leq \frac{1}{(a+1)(k+1)}
\]
and thus $\psi(x,y)\leq 1/(k+1)$ which yields $\psi\preceq^A\phi_n$ for the specified $A$.

\medskip

In the context of the varying distances $\phi_n$, we also need to consider the further requirement on $f$ that $\lim f(n)=\infty$. For this, we rely on the following quantitative notion of a rate of divergence:

\begin{definition}
We say that $\kappa$ is a rate of divergence for $f$ if for any $L\in\mathbb{N}$, we get
\[
\forall n\geq\kappa(L)\left(f(n)\geq L\right).
\]
\end{definition}

Note that if $f$ is nondecreasing, then $\kappa$ is already a rate of divergence if $f(\kappa(L))\geq L$.

\begin{theorem}\label{thm:quantThmFirst}
Let $(X,d)$ be a metric space and $(\phi_n)$ be a sequence of functions $\phi_n:X\times X\to\mathbb{R}_+$. Assume $\phi_n\nearrow\phi$ uniformly with a rate of convergence $\pi$ and that $\psi\preceq^A\phi_n$ for some $A:\mathbb{N}\to\mathbb{N}$ where $\psi$ is uniformly weakly triangular with modulus $\theta$. Assume that $X$ has a modulus of total $\phi$-boundedness $\gamma$, that $\alpha_G$ is a $G$-modulus for $G$ and $\beta_H$ is a $H$-modulus for $H$ and that $f$ is nondecreasing and $f(n)\leq n$. Further, let $(x_n)\subseteq X$ be a sequence and 
\begin{enumerate}
\item $\chi$ be a modulus for uniform $(\phi_{2n})$-$(G,H)$-quasi-Fej\'er monotonicity w.r.t.\ $F$ and $(AF_k)$ for $(x_{2n})$ with errors $(\varepsilon_{n})$,
\item $\xi$ be a Cauchy rate for $\sum\varepsilon_{n}<\infty$,
\item $\Phi$ be a $\liminf$-bound w.r.t.\ $F$ and $(AF_k)$ for $(x_{2f(n)})$,
\item $\kappa$ be a rate of divergence for $f$,
\item $\zeta$ be a modulus of uniform $(\phi_n)$-$(G,H)$-quasi-$f$-monotonicity w.r.t.\ $F$ and $(AF_k)$ for $(x_{n})$ with errors $(\varepsilon_n)$.
\end{enumerate}
Then $(x_n)$ is $\psi$-Cauchy with a rate of metastability $\Psi(k,g)$ (to be defined below), i.e.\ for all $k\in\mathbb{N},g\in\mathbb{N}^\mathbb{N}$:
\[
\exists N\leq\Psi(k,g)\forall i,j\in [N;N+g(N)]\left(\psi(x_i,x_j)\leq\frac{1}{k+1}\right).
\] 
If $\psi$ is uniformly triangular with modulus $\theta$, then $(x_n)$ is Cauchy with the same rate of metastability $\Psi(k,g)$.

Here: $\Psi(k,g)=2\Psi_0(P,k,g)$ with 
\[
P=\gamma(\max\{2\alpha_G(4\beta_H(A(\theta(k)))+3)+1,2\alpha_G(2\beta_H(\alpha_G(4\beta_H(A(\theta(k)))+3))+1)+1\})
\]
and $\Psi_0$ defined by
\[
\begin{cases}
\Psi_0(0,k,g)=0,\\
\Psi_0(n+1,k,g)=\Phi(\eta^M(\Psi_0(n,k,g),4\beta_H(A(\theta(k)))+3),\hat{k}),
\end{cases}
\]
where
\begin{align*}
\hat{k}=\max\{&\kappa(\xi(4\beta_H(\alpha_G(4\beta_H(A(\theta(k)))+3))+3)),\xi(2\beta_H(A(\theta(k)))+1),\kappa(\xi(2\beta_H(A(\theta(k)))+1)),\\
&\kappa(\pi(\max\{2\alpha_G(2\beta_H(\alpha_G(4\beta_H(A(\theta(k)))+3))+1)+1,2\alpha_G(4\beta_H(A(\theta(k)))+3)+1\}))\}
\end{align*}
and
\[
\eta(n,r)=\max\{\chi_g(n,r),\zeta_g(n,r),\chi\left(f(n),n-f(n),4\beta_H(\alpha_G(r))+3\right)\}
\]
with
\begin{gather*}
\chi_g(n,r)=\chi\left(n,\floor*{\frac{g(2n)}{2}},r\right),\\
\zeta_g(n,r)=\zeta\left(n,\floor*{\frac{g(2n)}{2}},r\right).
\end{gather*}
\end{theorem}
\begin{proof}
Set $\varphi(k,n)=\min\{m\geq n\mid x_{2f(m)}\in AF_k\}$ which is well-defined as $(x_{2f(n)})_n$ has the $\liminf$-property w.r.t.\ $F$. Let $(n_i)_i$ defined by $n_0=0$ and
\[
n_{i+1}=\varphi(\eta^M(n_i,4\beta_H(A(\theta(k)))+3),\hat{k}).
\]
Then for any $j\geq 1$ and any $0\leq i<j$, $x_{2f(n_j)}$ is an $\eta^M(n_i,4\beta_H(A(\theta(k)))+3)$-approximate $F$-point as by definition, $x_{2f(x_j)}$ is an $\eta^M(n_{j-1},4\beta_H(A(\theta(k)))+3)$-approximate-$F$-point and $n_i\leq n_{j-1}$ yields $\eta^M(n_{j-1},4\beta_H(A(\theta(k)))+3)\geq \eta^M(n_i,4\beta_H(A(\theta(k)))+3)$.

\medskip

Further, by the definition of $P$ and as $X$ is totally $\phi$-bounded with modulus $\gamma$, we have that $\exists 0<I<J\leq P$ (with $>0$ through the additional $+1$ in the definition of $P$) such that 
\[
\phi(x_{2f(n_J)},x_{2f(n_I)})\leq\frac{1}{2\alpha_G(2\beta_H(\alpha_G(4\beta_H(A(\theta(k)))+3))+1)+2}.
\]
As we have
\[
n_I\geq \hat{k}\geq \kappa(\pi(\max\{2\alpha_G(2\beta_H(\alpha_G(4\beta_H(A(\theta(k)))+3))+1)+1,2\alpha_G(4\beta_H(A(\theta(k)))+3)+1\}))
\]
this yields
\[
\phi_{2f(n_I)}(x_{2f(n_J)},x_{2f(n_I)})\leq\frac{1}{\alpha_G(2\beta_H(\alpha_G(4\beta_H(A(\theta(k)))+3))+1)+1}
\]
and thus we have
\[
G(\phi_{2f(n_I)}(x_{2f(n_J)},x_{2f(n_I)}))\leq\frac{1}{2\beta_H(\alpha_G(4\beta_H(A(\theta(k)))+3))+2}.
\]
By the above first result, we have that $x_{2f(n_J)}$ is an $\eta^M(n_I,4\beta_H(A(\theta(k)))+3)$-approximate $F$-point. Therefore, for
\[
l\leq n_I-f(n_I), 
\]
we get 
\[
H(\phi_{2f(n_I)+2l}(x_{2f(n_J)},x_{2f(n_I)+2l}))\leq\frac{1}{\beta_H(\alpha_G(4\beta_H(A(\theta(k)))+3))+1}.
\]
To see that, note that
\[
\eta^M(n_I,4\beta_H(A(\theta(k)))+3)\geq\chi\left(f(n_I),n_I-f(n_I),4\beta_H(\alpha_G(4\beta_H(A(\theta(k)))+3))+3\right)
\]
and as $x_{2f(n_J)}$ is an $\eta^M(n_I,4\beta_H(A(\theta(k)))+3)$-approximate $F$-point and $f(n_I)\geq\xi(4\beta_H(\alpha_G(4\beta_H(A(\theta(k)))+3))+3)$ since $n_I\geq\kappa(\xi(4\beta_H(\alpha_G(4\beta_H(A(\theta(k)))+3))+3))$, we get that 
\begin{align*}
&H(\phi_{2f(n_I)+2l}(x_{2f(n_J)},x_{2f(n_I)+2l}))\\
&\qquad\qquad\qquad= H(\phi_{2(f(n_I)+l)}(x_{2f(n_J)},x_{2(f(n_I)+l)}))\\
&\qquad\qquad\qquad\leq G(\phi_{2f(n_I)}(x_{2f(n_J)},x_{2f(n_I)}))+\sum_{i=f(n_I)}^{\infty}\varepsilon_i+\frac{1}{4\beta_H(\alpha_G(4\beta_H(A(\theta(k)))+3))+4}\\
&\qquad\qquad\qquad\leq \frac{1}{\beta_H(\alpha_G(4\beta_H(A(\theta(k)))+3))+1}.
\end{align*}
Now, as $f(n)\leq n$, we in particular have that $2f(n_I)\leq 2n_I$ and thus $2n_I=2f(n_I)+2l$ for $l=n_I-f(n_I)$. Thus we in particular have
\[
H(\phi_{2n_I}(x_{2f(n_J)},x_{2n_I}))\leq \frac{1}{\beta_H(\alpha_G(4\beta_H(A(\theta(k)))+3))+1}.
\]
Thus, in particular
\[
\phi_{2n_I}(x_{2f(n_J)},x_{2n_I})\leq\frac{1}{\alpha_G(4\beta_H(A(\theta(k)))+3)+1}
\]
as $\beta_H$ is an $H$-modulus and as $\alpha_G$ is a $G$-modulus, we get
\[
G(\phi_{2n_I}(x_{2f(x_J)},x_{2n_I}))\leq\frac{1}{4\beta_H(A(\theta(k)))+4}.
\]

\medskip

Now, also by definition of $P$, note that this same pair $(x_{2f(n_J)},x_{2f(n_I)})$ satisfies
\[
\phi(x_{2f(n_J)},x_{2f(n_I)})\leq\frac{1}{2\alpha_G(4\beta_H(A(\theta(k)))+3)+2}.
\]
As we have
\[
n_I\geq\hat{k}\geq \kappa(\pi(\max\{2\alpha_G(2\beta_H(\alpha_G(4\beta_H(A(\theta(k)))+3))+1)+1,2\alpha_G(4\beta_H(A(\theta(k)))+3)+1\})),
\]
this yields
\[
\phi_{2f(n_I)}(x_{2f(n_J)},x_{2f(n_I)})\leq\frac{1}{\alpha_G(4\beta_H(A(\theta(k)))+3)+1}
\]
and thus
\[
G(\phi_{2f(n_I)}(x_{2f(n_J)},x_{2f(n_I)}))\leq\frac{1}{4\beta_H(A(\theta(k)))+4}.
\]
We can then show that for any $l\leq g(2n_I)$
\[
H(\phi_{2n_I+l}(x_{2f(n_J)},x_{2n_I+l}))\leq\frac{1}{\beta_H(A(\theta(k)))+1}.
\]

\medskip

To see that, assume first that $l$ is even, i.e.\ $l=2l'$. Since
\[
\eta^M(n_I,r)\geq\chi\left(n_I,\floor*{\frac{g(2n_I)}{2}},r\right)
\]
and as $x_{2f(n_J)}$ is an $\eta^M(n_I,4\beta_H(A(\theta(k)))+3)$-approximate $F$-point and since
\[
l\leq g(2n_I)\text{ implies }l'\leq\floor*{\frac{g(2n_I)}{2}},
\]
we get that 
\begin{align*}
H(\phi_{2n_I+l}(x_{2f(n_J)},x_{2n_I+l}))&= H(\phi_{2(n_I+l')}(x_{2f(n_J)},x_{2(n_I+l')}))\\
&\leq G(\phi_{2n_I}(x_{2f(n_J)},x_{2n_I}))+\sum_{i=n_I}^\infty\varepsilon_i+\frac{1}{4\beta_H(A(\theta(k)))+4}\\
&\leq \frac{1}{\beta_H(A(\theta(k)))+1}
\end{align*}
where we have used $n_I\geq\xi(2\beta_H(A(\theta(k)))+1)$.

\medskip

Now, assume that $l$ is odd, i.e.\ $l=2l'+1$. Then, since
\[
\eta^M(n_I,r)\geq\zeta\left(n_I,\floor*{\frac{g(n_I)}{2}},r\right)
\]
and as $x_{2f(n_J)}$ is an $\eta^M(n_I,4\beta_H(A(\theta(k)))+3)$-approximate $F$-point and since
\[
l\leq g(2n_I)\text{ implies }l'\leq\floor*{\frac{g(2n_I)}{2}}
\]
as before, we get that 
\begin{align*}
H(\phi_{2n_I+l}(x_{2f(n_J)},x_{2n_I+l}))&= H(\phi_{2(n_I+l')+1}(x_{2f(n_J)},x_{2(n_I+l')+1}))\\
&\leq G(\phi_{2f(n_I)}(x_{2f(n_J)},x_{2f(n_I)}))+\sum_{i=f(n_I)}^\infty\varepsilon_i+\frac{1}{4\beta_H(A(\theta(k)))+4}\\
&\leq \frac{1}{\beta_H(A(\theta(k)))+1}
\end{align*}
where we have used $n_I\geq\kappa(\xi(2\beta_H(A(\theta(k)))+1))$ which implies $f(n_I)\geq\xi (2\beta_H(A(\theta(k)))+1)$.

\medskip

As $\beta_H$ is a $H$-modulus, we get
\[
\phi_{2n_I+l}(x_{2f(n_J)},x_{2n_I+l})\leq\frac{1}{A(\theta(k))+1}
\]
for all $l\leq g(2n_I)$ and thus
\[
\psi(x_{2f(n_J)},x_{2n_I+l})\leq \frac{1}{\theta(k)+1}
\]
for all $l\leq g(2n_I)$. Thus, for $m,l\in [2n_I,2n_I+g(2n_I)]$, we get
\[
\psi(x_m,x_l)\leq\frac{1}{k+1}
\]
which yields the result for $N=2n_I$. If $\psi$ is uniformly triangular with modulus $\theta$, then it similarly follows that for $m,l\in [2n_I,2n_I+g(2n_I)]$:
\[
d(x_m,x_l)\leq\frac{1}{k+1}.
\]
Note that we in any way have $n_i\leq\Psi_0(i,k,g)$ for all $i$ (using the monotonicity of $\Phi$) and thus $N\leq 2\Psi_0(I,k,g)\leq 2\Psi_0(P,k,g)$.
\end{proof}

This quantitative result can now be extended to include the moduli of uniform closedness for a stronger conclusion (similar to Theorem 5.3 of \cite{KLN2018}) that actually guarantees that the resulting limit of the sequence belongs to the solution set $F$.

\begin{theorem}\label{thm:quantThmSecond}
In addition to the assumptions of Theorem \ref{thm:quantThmFirst}, assume that $F$ is uniformly $\psi$-closed with moduli $\delta$ and $\omega$. Then for all $k\in\mathbb{N}$ and all $g\in \mathbb{N}^\mathbb{N}$:
\[
\exists N\leq\tilde{\Psi}(k,g)\forall i,j\in [N;N+g(N)]\left(\psi(x_i,x_j)\leq\frac{1}{k+1}\text{ and }x_i\in AF_k\right),
\] 
where $\tilde{\Psi}(k,g)=2\tilde{\Psi}_0(P,k,g)$ and where $\tilde{\Psi}_0$ is defined similar to $\Psi_0$ in Theorem \ref{thm:quantThmFirst} with 
\[
\tilde{\eta}_{k,\delta}(n,r)=\max\{\delta(k),\eta(n,r)\}
\]
instead of $\eta$ and with $\tilde{\theta}(k)=\max\{\theta(k),\omega(k)\}$ instead of $\theta$.

If $\psi$ is instead uniformly triangular with a modulus $\theta$, then for all $k\in\mathbb{N}$ and all $g\in \mathbb{N}^\mathbb{N}$:
\[
\exists N\leq\tilde{\Psi}(k,g)\forall i,j\in [N;N+g(N)]\left(d(x_i,x_j)\leq\frac{1}{k+1}\text{ and }x_i\in AF_k\right)
\] 
with the same $\tilde{\Psi}$.
\end{theorem} 
\begin{proof}
Note that $\tilde{\theta}$ is also a modulus of uniform weak triangularity for $\psi$. Thus, as in the proof of Theorem \ref{thm:quantThmFirst}, we obtain 
\[
\exists N\leq\tilde{\Psi}(k,g)\forall i,j\in [N;N+g(N)]\left( \psi(x_i,x_j)\leq\frac{1}{k+1}\right)
\]
and, more precisely, that there exists $0<I<J\leq P$ such that $N=2n_I$ and
\[
\psi(x_{2f(n_J)},x_{2n_I+l})\leq \frac{1}{\tilde{\theta}(k)+1}\leq\frac{1}{\omega(k)+1}
\]
for all $l\leq g(2n_I)$ and that $x_{2f(n_J)}$ is an $\tilde{\eta}^M_{k,\delta}(n_I,4\beta_H(A(\theta(k)))+3)$-approximate $F$-point. As $\tilde{\eta}^M_{k,\delta}(n_I,4\beta_H(A(\theta(k)))+3)\geq \delta(k)$, we have that $x_{2f(n_J)}$ is a $\delta(k)$-approximate $F$-point. Thus, we get $x_{2n_I+l}\in AF_k$ for any $l\leq g(n_I)$.

If $\theta$ is a modulus of uniform triangularity, then $\tilde{\theta}$ is also such a modulus and the resulting claim follows likewise.
\end{proof}

With a similar reasoning as in \cite{KLN2018}, in both results above it is actually sufficient to have a rate of metastability for the summability of the error sequence $(\varepsilon_{n})$ instead of a full Cauchy rate. Similarly, a rate of metastability for the uniform convergence $\phi_n\nearrow\phi$ would be sufficient but for simplicity, we omit this here. Note also that we only for simplicity assume that the error sequences for the quasi-Fej\'er monotonicity and the quasi-$f$-monotonicity coincide as in the case of different error sequences $(\varepsilon_n)$ and $(\varepsilon'_n)$, constructing $(\varepsilon_n+\varepsilon'_n)$ provides a common error sequence for which a Cauchy rate can immediately be constructed by combining the two separate Cauchy rates.\\

Already from Theorem \ref{thm:quantThmSecond}, we would now be able to infer a ``plain'' convergence result for generalized Fej\'er monotone sequences in the spirit of Proposition \ref{pro:convRes}. However, this result would only take place over totally $\phi$-bounded spaces and we would have to require that all mappings are $\phi$-continuous. As both notions could be rather hard to verify in practice, we refrain from spelling this out here and consider the main ``plain'' convergence result to be that contained in Theorem \ref{thm:mainRes} which infers the normal convergence for the sequence under an additional mild assumption connecting $\phi$ and the usual metric $d$.

\begin{remark}\label{rem:bdcomp}
In the context of $(\phi_n)$-Fej\'er monotone sequences w.r.t.\ a set $F$, it holds that
\[
\phi_{n+1}(z,x_{n+1})\leq \phi_n(z,x_n)\leq \dots\leq \phi_0(z,x_0)
\]
for $z\in F$ and any $n$. Now, if $\phi_n\to\phi$ holds uniformly, i.e. 
\[
\forall k\in\mathbb{N}\exists N\in\mathbb{N}\forall n\geq N\forall x,y\in X\left(\vert\phi_n(x,y)-\phi(x,y)\vert\leq \frac{1}{k+1}\right),
\] 
then for large enough $n$ we get $\phi(z,x_n)\leq\phi_0(z,x_0)+1$ and thus we get that $x_n$ is $\phi$-bounded. This naturally extends to the various extensions like including $G,H$ or errors as well as partiality.

In this case where the sequence $(x_n)$ is contained in a $\phi$-bounded set, i.e.\ 
\[
(x_n)\subseteq\{ y\in X\mid \phi(z,y)\leq\alpha\}
\]
for some $z\in X$ and $\alpha>0$, the above results naturally extend to spaces which are $\phi$-boundedly compact, i.e.\ where for any $z\in X$ and $\alpha>0$, the set
\[
\{ y\in X\mid \phi(z,y)\leq\alpha\}
\]
is $\phi$-compact.
\end{remark}

\subsection{Some special cases}

The first special case that we want to discuss is when all $\phi_n$ coincide with a single $\phi$. In that case, we can not only simplify the resulting bounds but, in the absence of errors in the Fej\'er monotonicity, we can even weaken the assumption of the $\liminf$-property to that of just approximate $F$-points. This would not have been possible before as, beyond the treatment of errors, the $\liminf$-property was needed to deal with the assumption that $\phi_n\nearrow\phi$.

\medskip

Concretely, instantiating the previous results, we get the following: 

\begin{theorem}\label{thm:metaSingleDist}
Let $(X,d)$ be a metric space and $\phi$ be a function $\phi:X\times X\to\mathbb{R}_+$. Assume that $\psi\preceq^A\phi$ for some $A:\mathbb{N}\to\mathbb{N}$ where $\psi$ is uniformly weakly triangular with modulus $\theta$. Assume that $X$ has a modulus of total $\phi$-boundedness $\gamma$, that $\alpha_G$ is a $G$-modulus for $G$ and $\beta_H$ is a $H$-modulus for $H$ and that $f$ is nondecreasing and $f(n)\leq n$. Further, let $(x_n)\subseteq X$ be a sequence and 
\begin{enumerate}
\item $\chi$ be a modulus for uniform $\phi$-$(G,H)$-quasi-Fej\'er monotonicity w.r.t.\ $F$ and $(AF_k)$ for $(x_{2n})$ with errors $(\varepsilon_{n})$,
\item $\xi$ be a Cauchy rate for $\sum\varepsilon_{n}<\infty$,
\item $\Phi$ be a $\liminf$-bound w.r.t.\ $F$ and $(AF_k)$ for $(x_{2f(n)})$,
\item $\kappa$ be a rate of divergence for $f$,
\item $\zeta$ be a modulus of uniform $\phi$-$(G,H)$-quasi-$f$-monotonicity w.r.t.\ $F$ and $(AF_k)$ for $(x_n)$ with errors $(\varepsilon_n)$.
\end{enumerate}
Then $(x_n)$ is $\psi$-Cauchy with a rate of metastability $\Psi(k,g)$ (to be defined below), i.e.\ for all $k\in\mathbb{N},g\in\mathbb{N}^\mathbb{N}$:
\[
\exists N\leq\Psi(k,g)\forall i,j\in [N;N+g(N)]\left(\psi(x_i,x_j)\leq\frac{1}{k+1}\right).
\]
If $\psi$ is uniformly triangular with modulus $\theta$, then $(x_n)$ is Cauchy with the same rate of metastability $\Psi(k,g)$.

Here: $\Psi(k,g)=2\Psi_0(P,k,g)$ with $\Psi_0$ defined as in Theorem \ref{thm:quantThmFirst}, now with 
\begin{align*}
\hat{k}=\max\{&\kappa(\xi(4\beta_H(\alpha_G(4\beta_H(A(\theta(k)))+3))+3)),\xi(2\beta_H(A(\theta(k)))+1),\kappa(\xi(2\beta_H(A(\theta(k)))+1))\}.
\end{align*}

If in addition to the above assumptions, $F$ is uniformly $\psi$-closed with modulus $\delta$ and $\omega$, then for all $k\in\mathbb{N}$ and all $g\in \mathbb{N}^\mathbb{N}$:
\[
\exists N\leq\tilde{\Psi}(k,g)\forall i,j\in [N;N+g(N)]\left(\psi(x_i,x_j)\leq\frac{1}{k+1}\text{ and }x_i\in AF_k\right),
\] 
where $\tilde{\Psi}(k,g)=2\tilde{\Psi}_0(P,k,g)$ and where $\tilde{\Psi}_0$ is defined similar to $\Psi_0$ before but with
\[
\tilde{\eta}_{k,\delta}(n,r)=\max\{\delta(k),\eta(n,r)\}
\]
instead of $\eta$ and with $\tilde{\theta}(k)=\max\{\theta(k),\omega(k)\}$ instead of $\theta$. As before, if $\psi$ is actually uniformly triangular with a modulus $\theta$, then for all $k\in\mathbb{N}$ and all $g\in \mathbb{N}^\mathbb{N}$:
\[
\exists N\leq\tilde{\Psi}(k,g)\forall i,j\in [N;N+g(N)]\left(d(x_i,x_j)\leq\frac{1}{k+1}\text{ and }x_i\in AF_k\right),
\] 

Lastly, if in any of the cases, instead of (1) and (5) we even have
\begin{enumerate}
\item [(1)'] $\chi$ is a modulus for uniform $\phi$-$(G,H)$-Fej\'er monotonicity w.r.t.\ $F$ and $(AF_k)$ for $(x_{2n})$,
\item [(5)'] $\zeta$ is a modulus of uniform $\phi$-$(G,H)$-$f$-monotonicity w.r.t.\ $F$ and $(AF_k)$ for $(x_n)$,
\end{enumerate}
then $\xi$ and $\kappa$ from assumptions (2) and (4) can be omitted and it suffices to assume that
\begin{enumerate}
\item [(3)'] $\Phi$ is an approximate $F$-point bound w.r.t.\ $(AF_k)$ for $(x_{2f(n)})$,
\end{enumerate}
and we get the same conclusions, respectively, with a rate $\hat{\Psi}(k,g)=2\hat{\Psi}_0(P,k,g)$ and where $\hat{\Psi}_0$ is defined as 
\[
\begin{cases}
\hat{\Psi}_0(0,k,g)=0,\\
\hat{\Psi}_0(n+1,k,g)=\Phi(\eta^M(\hat{\Psi}_0(n,k,g),4\beta_H(A(\theta(k)))+3)),
\end{cases}
\]
with the other constants defined as before.
\end{theorem}

The proof is immediate and we thus omit it.\\

Further, we now want to indicate how this result encompasses the previous quantitative results for plain Fej\'er monotone sequences. For that, we first consider how a modulus of uniform Fej\'er monotonicity for the full sequence actually can be used to construct the other relevant moduli for the uniform partial Fej\'er monotonicity.

\begin{proposition}
Suppose $(x_n)$ is uniformly $(\phi_n)$-$(G,H)$-quasi-Fej\'er monotone w.r.t.\ $F$ and $(AF_k)$ with errors $(\varepsilon_n)$ and a modulus $\chi$. Define $\widetilde{\varepsilon}_n=\varepsilon_{2n}+\varepsilon_{2n+1}$. Then $(x_{2n})$ is uniformly $(\phi_{2n})$-$(G,H)$-quasi-Fej\'er monotone w.r.t.\ $F$ and $(AF_k)$ with errors $(\widetilde{\varepsilon}_n)$ and a modulus
\[
\chi'(n,m,r)=\chi(2n,2m,r)
\]
and $(x_n)$ is uniformly $(\phi_n)$-$(G,H)$-quasi-$\mathrm{id}$-monotone w.r.t.\ $F$ and $(AF_k)$ with errors $(\widetilde{\varepsilon}_n)$ and a modulus
\[
\eta(n,m,r)=\chi(2n,2m+1,r).
\]
\end{proposition}
We omit the proof as it is rather immediate.

\medskip

Up to this point, we assumed the functions $G$ and $H$ to be the same in both the Fej\'er and the $f$-monotonicity property. This can be immediately extended to the case where the respective functions do not match which we will establish through the following results:

\begin{proposition}
Let $G_{1,2},H_{1,2}$ be given. Set $H=\min\{H_1,H_2\}$ and $G=\max\{G_1,G_2\}$. Let $\alpha_{G_i}$ be $G$-moduli for $G_i$ and $\beta_{H_i}$ be $H$-moduli for $H_i$. Then
\[
\alpha_G(k)=\max\{\alpha_{G_1}(k),\alpha_{G_2}(k)\}
\]
is a $G$-modulus for $G$ and
\[
\beta_H(k)=\max\{\beta_{H_1}(k),\beta_{H_2}(k)\}
\]
is a $H$-modulus for $H$.
\end{proposition}
\begin{proof}
Let $a\leq\frac{1}{\alpha_G(k)+1}\leq\frac{1}{\alpha_{G_i}(k)+1}$ for $i=1,2$. Thus $G_i(a)\leq\frac{1}{k+1}$ for $i=1,2$ and thus $G(a)\leq\frac{1}{k+1}$.

Similarly, let $H(a)\leq\frac{1}{\beta_H(k)+1}\leq\frac{1}{\beta_{H_i}(k)+1}$ for $i=1,2$. Then $H_i(a)\leq\frac{1}{\beta_{H_i}(k)+1}$ for some $i$. Thus $a\leq\frac{1}{k+1}$.
\end{proof}

\begin{proposition}
Let $(\phi_n)$ be a sequence of functions $\phi_n:X\times X\to\mathbb{R}_+$. Let $G_{1,2},H_{1,2}$ be given. Set $H=\min\{H_1,H_2\}$ and $G=\max\{G_1,G_2\}$. Then:
\begin{enumerate}
\item If $\chi$ is a modulus for uniform $(\phi_n)$-$(G_1,H_1)$-quasi-Fej\'er monotonicity of $(x_n)$ w.r.t.\ errors $(\varepsilon_n)$, then it is also a modulus for uniform $(\phi_n)$-$(G,H)$-quasi-Fej\'er monotonicity of $(x_n)$ w.r.t.\ $(\varepsilon_n)$.
\item If $\zeta$ is a modulus of uniform $(\phi_n)$-$(G_2,H_2)$-quasi-$f$-monotonicity of $(x_n)$ w.r.t.\ $(\varepsilon_n)$, then it is also a modulus of uniform $(\phi_n)$-$(G,H)$-quasi-$f$-monotonicity of $(x_n)$ w.r.t.\ $(\varepsilon_n)$.
\end{enumerate} 
\end{proposition}
We again omit the proof as it is rather immediate.

\medskip

\begin{proposition}
Let $(\phi_n)$ be a sequence of functions $\phi_n:X\times X\to\mathbb{R}_+$. Let $G_{1,2},H_{1,2}$ be given and assume that $G_2(x)\leq H_1(x)$ for all $x$. Let $f:\mathbb{N}\to\mathbb{N}$ be nondecreasing and $f(n)\leq n$. 

Let $\chi$ be a modulus of uniform $(\phi_{2n})$-$(G_1,H_1)$-quasi-Fej\'er monotonicity for $(x_{2n})$ w.r.t.\ $F$ and $(AF_k)$ with errors $(\varepsilon_{n})$ and let $\zeta:\mathbb{N}^2\to\mathbb{N}$ be nondecreasing in the left argument such that for all $n,r\in\mathbb{N}$, there exists a $k\leq\zeta(n,r)$ such that
\[
\forall p\in AF_k\left(H_2(\phi_{2n+1}(p,x_{2n+1}))\leq G_2(\phi_{2f(n)}(p,x_{2f(n)}))+\sum_{i=f(n)}^{n}\varepsilon_i+\frac{1}{r+1}\right).
\]
Then $\hat{\zeta}(n,m,r)=\max\{\chi(f(n),f(n+m)-f(n),2r+1),\zeta(n+m,2r+1)\}$ is a modulus of uniform $(\phi_n)$-$(G_1,H_2)$-quasi-$f$-monotonicity for $(x_n)$ w.r.t.\ $F$ and $(AF_k)$ with errors $(\varepsilon_n)$.
\end{proposition}
\begin{proof}
Let $k=\hat{\zeta}(n,m,r)$ and $p\in AF_k$. Then $k\geq\zeta(n+m,2r+1)\geq\zeta(n+l,2r+1)$ and thus
\[
H_2(\phi(p,x_{2(n+l)+1}))\leq G_2(\phi(p,x_{2f(n+l)}))+\sum_{i=f(n+l)}^{n+l}\varepsilon_i+\frac{1}{2r+1}.
\]
Using $G_2\leq H_1$ and $k\geq\chi(n,m,2r+1)$, the above then yields
\begin{align*}
H_2(\phi(p,x_{2(n+l)+1}))&\leq H_1(\phi(p,x_{2f(n+l)}))+\sum_{i=f(n+l)}^{n+l}\varepsilon_i+\frac{1}{2r+1}\\
&\leq G_1(\phi(p,x_{2f(n)}))+\sum_{i=f(n+l)}^{n+l}\varepsilon_i+\sum_{i=f(n)}^{f(n+l)-1}\varepsilon_i+\frac{1}{r+1}\\
&=G_1(\phi(p,x_{2f(n)}))+\sum_{i=f(n)}^{n+l}\varepsilon_i+\frac{1}{r+1}
\end{align*}
as $f(n)\leq f(n+l)$ since $f$ is nondecreasing.
\end{proof}

Lastly, we want to consider a special situation of the $f$-monotonicity property. Assume that we are in a situation where
\[
H(\phi_{2(n+m)+s}(p,x_{2(n+m)+s}))\leq G(\phi_{2n}(p,x_{2n}))
\]
for any $n,m\in\mathbb{N}$ and any $p\in F$ where $s=2s'+1$ is a given odd step length. Such a situation can be converted to the above setting in a straightforward way:
\begin{proposition}\label{pro:affineLinearF}
Let $(\phi_n)$ be a sequence of functions $\phi_n:X\times X\to\mathbb{R}_+$ and let $\chi$ be a modulus of uniform $(\phi_{2n})$-$(G,H)$-quasi-Fej\'er monotonicity for $(x_{2n})$ w.r.t.\ $F$ and $(AF_k)$ with errors $(\varepsilon_n)$. Assume $s=2s'+1$ and let $\zeta:\mathbb{N}^3\to\mathbb{N}$ be such that for all $n,m,r\in\mathbb{N}$, there exists a $k\leq\zeta(n,m,r)$ such that for all $p\in AF_k$:
\[
\forall l\leq m\left(H(\phi_{2(n+l)+s}(p,x_{2(n+l)+s}))\leq G(\phi_{2n}(p,x_{2n}))+\sum_{i=f(n)}^{n+l+s'}\varepsilon_i+\frac{1}{r+1}\right).
\]
Then, for $f(n)=n\dotdiv s'$, $\hat{\zeta}(n,m,r)=\max\{\zeta(n,m,r),\chi(0,s',r)\}$ is a modulus of uniform $(\hat{\phi}_n)$-$(G,H)$-quasi-$f$-monotonicity for $(\hat{x}_n)$ w.r.t.\ $F$ and $(AF_k)$ with errors $(\varepsilon_n)$ where for any sequence $(a_n)$ of objects, we define
\[
\hat{a}_n=\begin{cases}
a_n&\text{for }n\geq s,\\
a_n&\text{for even }n< s,\\
a_{n-1}&\text{for odd }n<s.
\end{cases}
\]
\end{proposition}
So, we in particular see that the above situation is covered already via affine linear $f$. In that case, note that a modulus of uniform Fej\'er monotonicity for the sequence $(x_{2n})$ is also a modulus of uniform Fej\'er monotonicity for $(\hat{x}_{2n})$. Similarly, we get the following result:
\begin{proposition}\label{pro:affineLinearFapprox}
Assume $s=2s'+1$ and let $\Phi$ be an approximate $F$-point bound for $(x_{2n})$. Then $\Phi'(k)=\Phi(k)+s'$ is an approximate $F$-point bound for $(x_{2f(n)})$ where $f(n)=n\dotdiv s'$.
\end{proposition}

\section{Relation to the metric case and to Bregman distances}\label{sec:ConsDist}

\subsection{Consistent distances and their properties}

We now want to discuss the relation of the previously introduced relativized notions to their original metric counterparts if the distance $\phi$ in question is suitably related to the metric. Concretely, in many cases\footnote{In fact, in the important special case of Bregman distances (which will be later discussed) this property was singled out as one of the central properties of the abstract generalized distances considered by Bregman already in the seminal paper \cite{Bre1967} (see Condition I on p.\ 200  therein) which formed the basis for the concrete Bregman distances now prevalent in the literature.} one works with distances $\phi$ which satisfy
\[
\phi(x,y)=0\leftrightarrow d(x,y)=0.
\]

In the following, we will consider distances that fulfill the above property quantitatively in the sense of the following definition:
\begin{definition}
We call $\phi$ consistent if there exist two functions $\lambda,\Lambda:\mathbb{N}\times\mathbb{N}\to\mathbb{N}$, called moduli of consistency, that satisfy
\begin{gather*}
\phi(x,y)\leq\frac{1}{\lambda(k,b)+1}\to d(x,y)\leq\frac{1}{k+1},\\
d(x,y)\leq\frac{1}{\Lambda(k,b)+1}\to \phi(x,y)\leq\frac{1}{k+1},
\end{gather*}
for any $k,b\in\mathbb{N}$ and for any $x,y\in X$ such that $d(a,x),d(a,y)\leq b$ holds w.r.t.\ some reference point $a\in X$.\footnote{As before with the modulus $\theta$, the existence of such moduli does not depend on the concrete reference point.} We call such functions moduli of uniform consistency (and $\phi$ uniformly consistent) if they do not depend on $b$.
\end{definition}

\begin{remark}
Note that given $\lambda,\Lambda:\mathbb{N}\to\mathbb{N}$, these functions are moduli of uniform consistency for $\phi$ if, and only if $d\preceq^\lambda\phi$ and $\phi\preceq^\Lambda d$.
\end{remark}

We will later discuss an example of a consistent distance beyond the metric.\\

In the context of such moduli of consistency, we can immediately derive that $\phi$ is weakly triangular as well as triangular and derive corresponding moduli from $\lambda,\Lambda$: Define $\theta(k,b)=\lambda(2k+1,b)$. Then for $x,y,z$ such that $d(a,x),d(a,y),d(a,z)\leq b$, assume that 
\[
\phi(y,x),\phi(y,z)\leq\frac{1}{\theta(k,b)+1}.
\]
Then in particular
\[
d(y,x),d(y,z)\leq\frac{1}{2k+2}
\]
by assumption on $\lambda$ and thus $d(x,z)\leq 1/(k+1)$. Thus $\theta$ is a modulus of triangularity. Clearly, it also follows immediately from this that $\theta(k,b)=\lambda(2\Lambda(k,b)+1,b)$ is a modulus of weak triangularity.

\medskip

Further, a simple compactness argument shows again that if the space is compact and $\phi$ is continuous, then the above property
\[
\phi(x,y)=0\leftrightarrow d(x,y)=0
\]
already implies the existence of moduli of uniform consistency for $\phi$.

\begin{remark}[For logicians]
As in the context of the moduli of (weak) triangularity, also here the underlying logical methods from proof mining actually guarantee that corresponding moduli $\lambda,\Lambda$ can be extracted from a large class of proofs of the property
\[
\phi(x,y)=0\leftrightarrow d(x,y)=0,
\]
even in the absence of any compactness assumptions.
\end{remark}

Even further, using such moduli, we can translate corresponding moduli of total $\phi$-boundedness and uniform $\phi$-closedness into moduli for their metric counterparts, and vice versa. For that, we again rely on these uniform versions of the moduli $\lambda,\Lambda$ which in particular can be derived from the usual moduli in any bounded metric space. As the proof is rather straightforward, we omit it.

\begin{proposition}
Let $\phi:X\times X\to\mathbb{R}_+$ be a uniformly consistent mapping with moduli $\lambda,\Lambda:\mathbb{N}\to\mathbb{N}$. Then:
\begin{enumerate}
\item If $\gamma$ is a modulus of total boundedness for some set $A\subseteq X$, then $\gamma'(k)=\gamma(\Lambda(k))$ is a modulus of total $\phi$-boundedness.
\item If $\omega,\delta$ are moduli of uniform closedness for a given $F$ and a sequence $(AF_k)$, then $\omega'(k)=\lambda(\omega(k))$ and $\delta'(k)=\delta(k)$ are respective moduli of uniform $\phi$-closedness.
\end{enumerate}
\end{proposition}

Derived from these conversions, we can now obtain the following modified general quantitative convergence result where the assumptions are phrased in terms of the metric (and thus are witnessed by moduli already commonly used in the literature).

\begin{theorem}\label{thm:metastabfversionmetricversion}
Let $(X,d)$ be a metric space and $(\phi_n)$ be a sequence of functions $\phi_n:X\times X\to\mathbb{R}_+$. Assume that $\phi:X\times X\to\mathbb{R}_+$ is uniformly consistent with moduli $\lambda,\Lambda:\mathbb{N}\to\mathbb{N}$ and satisfies $\phi_n\nearrow\phi$ uniformly with a rate of convergence $\pi$. Further assume that $\psi\preceq^A\phi_n$ for some $A:\mathbb{N}\to\mathbb{N}$ where $\psi$ is uniformly weakly triangular with a modulus $\theta$. Assume that $X$ is totally bounded with a modulus of total boundedness $\gamma$, $\alpha_G$ be a $G$-modulus for $G$, $\beta_H$ be a $H$-modulus for $H$ and $f:\mathbb{N}\to\mathbb{N}$ be nondecreasing and $f(n)\leq n$. Further, let $(x_n)\subseteq X$ be a sequence and
\begin{enumerate}
\item $\chi$ be a modulus for uniform $(\phi_{2n})$-$(G,H)$-quasi-Fej\'er monotonicity w.r.t.\ $F$ and $(AF_k)$ for $(x_{2n})$ with errors $(\varepsilon_n)$,
\item $\xi$ be a Cauchy rate for $\sum\varepsilon_{n}<\infty$,
\item $\Phi$ be a $\liminf$-bound w.r.t.\ $F$ and $(AF_k)$ for $(x_{2f(n)})$,
\item $\kappa$ be a rate of divergence for $f$,
\item $\zeta$ be a modulus of uniform $(\phi_n)$-$(G,H)$-quasi-$f$-monotonicity w.r.t.\ $F$ and $(AF_k)$ for $(x_{n})$ with errors $(\varepsilon_n)$.
\end{enumerate}
Then $(x_n)$ is $\psi$-Cauchy and for all $k\in\mathbb{N},g\in\mathbb{N}^\mathbb{N}$:
\[
\exists N\leq\Psi'(k,g)\forall i,j\in [N;N+g(N)]\left( \psi(x_i,x_j)\leq\frac{1}{k+1}\right),
\] 
where $\Psi'(k,g)$ is defined similar to $\Psi$ from Theorem \ref{thm:quantThmFirst}, now using $\gamma'(k)=\gamma(\Lambda(k))$. As before, if $\psi$ is uniformly triangular with a modulus $\theta$, then the sequence is Cauchy with the same rate of metastability.

If in addition to the above assumptions, $F$ is uniformly $\psi$-closed with moduli $\delta$ and $\omega$, then for all $k\in\mathbb{N}$ and all $g\in \mathbb{N}^\mathbb{N}$:
\[
\exists N\leq\tilde{\Psi'}(k,g)\forall i,j\in [N;N+g(N)]\left( d(x_i,x_j)\leq\frac{1}{k+1}\text{ and }x_i\in AF_k\right),
\] 
where $\tilde{\Psi'}(k,g)$ is defined similar to $\Psi'$ from above but with
\[
\tilde{\eta}_{k,\delta}(n,r)=\max\{\delta(k),\eta(n,r)\}
\]
instead of $\eta$ and with $\tilde{\theta}(k)=\max\{\theta(k),\omega(k)\}$ instead of $\theta$.

If $\psi$ is also uniformly consistent with moduli $\lambda',\Lambda':\mathbb{N}\to\mathbb{N}$, then it suffices to assume that $F$ is uniformly closed with moduli $\delta$ and $\omega$ and the above rates apply with $\delta'(k)=\delta(k)$ and $\omega'(k)=\lambda(\omega(k))$ used instead.

Lastly, if in any of the cases, we have $\phi_n=\phi$ for all $n$ and instead of (1) and (5) we even have
\begin{enumerate}
\item [(1)'] $\chi$ is a modulus for uniform $\phi$-$(G,H)$-Fej\'er monotonicity w.r.t.\ $F$ and $(AF_k)$ for $(x_{2n})$,
\item [(5)'] $\zeta$ is a modulus of uniform $\phi$-$(G,H)$-$f$-monotonicity w.r.t.\ $F$ and $(AF_k)$ for $(x_n)$,
\end{enumerate}
then $\xi$ and $\kappa$ from assumptions (2) and (4) can be omitted and it suffices to assume that
\begin{enumerate}
\item [(3)'] $\Phi$ is an approximate $F$-point bound w.r.t.\ $(AF_k)$ for $(x_{2f(n)})$,
\end{enumerate}
and we get the same conclusions, respectively, with a rate $\hat{\Psi}(k,g)=2\hat{\Psi}_0(P,k,g)$ where $\hat{\Psi}_0$ is defined as 
\[
\begin{cases}
\hat{\Psi}_0(0,k,g)=0,\\
\hat{\Psi}_0(n+1,k,g)=\Phi(\eta^M(\hat{\Psi}_0(n,k,g),4\beta_H(A(\theta(k)))+3)),
\end{cases}
\]
with the other constants defined as before.
\end{theorem}

Again, the proof is immediate and we thus omit it.

\begin{remark}
Similar to Remark \ref{rem:bdcomp}, also these results hold in the more general case of boundedly compact spaces if the sequence in question is bounded. However, in contrast to Fej\'er monotonicity w.r.t.\ the metric $d$, the property of a sequence $(x_n)$ being Fej\'er monotone w.r.t.\ $(\phi_n)$ even for $\phi_n\to\phi$ does not suffice to infer that said sequence is bounded but only that it is $\phi$-bounded as discussed before in Remark \ref{rem:bdcomp}. This is effectively due to the fact that we know almost nothing about $\phi$. In fact, the property that $(\phi_n)$-Fej\'er monotonicity implies boundedness of the sequence is not even guaranteed for the special case of Bregman monotone sequences, i.e.\ where $\phi_n=D_f$ for any $n$. In that context, it is thus a common requirement in the literature (being contained in various notions of a ``Bregman function'', see e.g.\ \cite{Eck1993}, and in the stronger form of requiring compactness already being mentioned in Bregman's work \cite{Bre1967} for his general class of distances $D$) to require that the level sets
\begin{gather*}
L_1(y,\alpha) = \{x\in X\mid D_f(x,y)\leq\alpha\},\\
L_2(x,\alpha) = \{y\in X\mid D_f(x,y)\leq\alpha\},
\end{gather*}
are bounded for every $\alpha>0$ and $x,y\in X$.

In similarity, we thus may additionally require that our function $\phi$ satisfies the following boundedness condition:
\[
\{y\in X\mid \phi(z,y)\leq\alpha\}\text{ is bounded for any }z\in X,\alpha>0.
\]
Then the $\phi$-boundedness translates into the boundedness of $(x_n)$ and thus the above results holds in boundedly compact spaces in that context.
\end{remark}

This result now gives rise to the following general convergence result by ``forgetting'' about the quantitative data in the assumptions and conclusions. Concretely we obtain the following generalization of Proposition \ref{pro:convRes}.

\begin{theorem}\label{thm:mainRes}
Let $(X,d)$ be compact and $F\subseteq X$ be explicitly closed. Let $(\phi_n)$ be a sequence of functions $\phi_n:X\times X\to\mathbb{R}_+$. Assume $\phi_n\nearrow\phi$ uniformly and that $\psi\preceq^A\phi_n$ where $A:\mathbb{N}\to\mathbb{N}$ and where $\psi,\phi:X\times X\to\mathbb{R}_+$ are uniformly consistent as well as continuous in their left arguments. If $(x_n)\subseteq X$ is such that there exists a nondecreasing function $f:\mathbb{N}\to\mathbb{N}$ with $f(n)\leq n$ and $f(n)\to\infty$ for $n\to\infty$ and where $(x_n)$ is partially $(\phi_n)$-$(G,H)$-quasi-Fej\'er monotone w.r.t.\ $F$, $f$ and $(\varepsilon_n)$ and $(x_{2f(n)})$ has the $\liminf$-property w.r.t.\ $F$ and $(AF_k)$, then $(x_n)$ converges to some $x^*\in F$.

In particular, the above holds if $\psi$ and $\phi$, instead of being uniformly consistent, are continuous also in their right arguments and it additionally holds that
\[
\psi(x,y)=0\leftrightarrow d(x,y)=0\leftrightarrow\phi(x,y)=0
\] 
for all $x,y\in X$.

If $\phi_n=\phi$ and $\varepsilon_n=0$ for all $n$, then it suffices that $(x_{2f(n)})$ has approximate $F$-points w.r.t.\ $(AF_k)$.
\end{theorem}
\begin{proof}
The result follows from Theorem \ref{thm:metastabfversionmetricversion}, noting that all the relevant uniform premises follow from the nonuniform variants by simple compactness arguments as mentioned before and by noting that metastability implies convergence.
\end{proof}

This result not only generalizes the previous Proposition \ref{pro:convRes} but also encompasses (special cases of) other results on convergence for ``Fej\'er-like monotone'' sequences. As a representative example, besides those discussed in subsequent papers, we in particular want to mention the work of \cite{BBC2003} where a convergence result (see Theorem 4.11 therein) was obtained for sequences which are Fej\'er monotone w.r.t.\ Bregman-distances 
\[
D_f:X\times X\to [0,+\infty],\quad (x,y)\mapsto\begin{cases}f(x)-f(y)-\langle x-y,\nabla f(y)\rangle&\text{if }y\in\mathrm{int}\mathrm{dom}f,\\+\infty&\text{otherwise},\end{cases}
\]
for $f:X\to (0,+\infty]$ a lower semicontinuous convex function which is G\^{a}teaux-differentiable on $\mathrm{int}\mathrm{dom}f\neq\emptyset$ over a reflexive Banach space $X$. In particular, in the context of the strong convergence result presented there, the additional requirement (called Condition 4.4 in \cite{BBC2003}) that
\[
D_f(x_n,y_n)\to 0\text{ implies } x_n-y_n\to 0
\]
for all bounded sequences $(x_n)$ and $(y_n)$ in $\mathrm{int}\mathrm{dom}f$ is discussed. If $f$ is real-valued (i.e.\ $f(x)\neq+\infty$ for all $x$) then the above clearly yields that $D_f(x,y)=0$ implies $\norm{x-y}=0$. If $f$ is strictly convex, the Bregman distance $D_f$ also naturally satisfies that $\norm{x-y}=0$ implies $D_f(x,y)=0$ and further, if $f$ is even Fr\'echet-differentiable, then $D_f$ is continuous in both arguments. In that vein, our above convergence results covers these partial cases of the result obtained in \cite{BBC2003} if the sequence further belongs to a compact subset (as is e.g.\ in particular the case if the space is finite dimensional as the sequences are naturally bounded).

\subsection{A more involved example of a uniformly consistent distance}

We are now concerned with a more involved particular example of a concrete Bregman distances where moduli of consistency can be naturally constructed. Now, for general Bregman distances $D_f$, the situation is in general quite complicated as not all distances are immediately consistent. We refer to \cite{KP2023b} for further discussions on (the use of) moduli of consistency in this general case and in this subsection only survey the special case $D_f$ of the Bregman distance for $f=\norm{\cdot}^2/2$ in the context of uniformly convex and uniformly smooth Banach spaces where such moduli can be readily obtained (although this is non-trivial). Instead of working with $D_f$ directly, we follow the setup of e.g.\ the works \cite{Alb1996,KKT2004,KT2004} and rather work with the function $\phi$ defined by\footnote{It can be easily seen via the use of the Fenchel conjugate of $f$ for the above $f$ that $\phi$ is nothing else than $2D_f$.} 
\[
\phi(x,y)=\norm{x}^2-2\langle x,Jy\rangle+\norm{y}^2
\]
where $J$ is the (single-valued) duality map
\[
J(x)=\{x^*\in X^*\mid \norm{x}^2=\norm{x^*}^2=\langle x,x^*\rangle\}.
\]
Basic considerations on $\phi$ immediately yield that
\[
0\leq\left(\norm{x}-\norm{y}\right)^2\leq\phi(x,y).
\]
Central for obtaining moduli of consistency for this distance $\phi$ is the following result due to Alber:
\begin{proposition}[Alber \cite{Alb1996}]
Let $X$ be a uniformly convex and uniformly smooth Banach space with a modulus of convexity $\delta_X$, i.e.
\[
\delta_X(\varepsilon)=\inf\{1-\norm{x+y}/2\mid x,y\in X, \norm{x},\norm{y}\leq 1,\norm{x-y}\geq \varepsilon\},
\]
and a modulus of uniform smoothness $\rho_X$, i.e. 
\[
\rho_X(\tau)=\sup\{1/2(\norm{x+y}+\norm{x-y}-2)\mid x,y\in X,\norm{x}=1,\norm{y}\leq\tau\}.
\]
Then for all $x,y\in X$:
\[
8C^2\delta_X\left(\norm{x-y}/4C\right)\leq\phi(x,y)\leq 4C^2\rho_X\left(4\norm{x-y}/C\right)
\]
where
\[
C=\sqrt{(\norm{x}^2+\norm{y}^2)/2}.
\]
In particular, utilizing the Figiel inequalities \cite{Fig1976}, for $\norm{x},\norm{y}\leq b$:
\[
2L^{-1}b^2\delta_X\left(\norm{x-y}/4b\right)\leq\phi(x,y)\leq 4L^{-1}b^2\rho_X\left(4\norm{x-y}/b\right)
\]
where $1<L<3.18$ is the constant of the Figiel inequalities \cite{Fig1976}.
\end{proposition}
From this, suitable moduli $\lambda,\Lambda$ can be immediately derived.
\begin{lemma}\label{lem:phimoduli}
Let $X$ be a uniformly convex and uniformly smooth Banach space and let $\eta:(0,2]\to (0,1]$ be a nondecreasing modulus of uniform convexity for $X$, i.e.
\[
\forall\varepsilon\in (0,2]\forall x,y\in X\left( \norm{x},\norm{y}\leq 1\land \norm{x-y}\geq\varepsilon\rightarrow\norm{\frac{x+y}{2}}\leq 1-\eta(\varepsilon)\right).
\]
Then, the functions
\begin{gather*}
\tilde{\Lambda}(\varepsilon,b)=\frac{\varepsilon}{16L^{-1}b},\\
\tilde{\lambda}(\varepsilon,b)=\frac{b^2}{L}\eta\left(\frac{\varepsilon}{4b}\right),
\end{gather*}
satisfy
\begin{gather*}
\norm{x-y}\leq\tilde{\Lambda}(\varepsilon,b)\to \phi(x,y)\leq\varepsilon,\\
\phi(x,y)\leq\tilde{\lambda}(\varepsilon,b)\to\norm{x-y}\leq\varepsilon,
\end{gather*}
for all $x,y\in X$ with $\norm{x},\norm{y}\leq b$. In particular, the functions
\begin{gather*}
\Lambda(k,b)=\ceil*{\frac{k+1}{16L^{-1}b}}\dotdiv 1,\\
\lambda(k,b)=\ceil*{\frac{L}{b^2}\left(\eta\left(\frac{1}{4b(k+1)}\right)\right)^{-1}}\dotdiv 1,
\end{gather*}
are moduli of consistency for $\phi$.
\end{lemma}
\begin{proof}
Note that $\eta(\varepsilon)\leq\delta_X(\varepsilon)$ and that $\rho_X(\varepsilon)\leq\varepsilon$.  Therefore, for one, if
\[
\norm{x-y}\leq\tilde{\Lambda}(\varepsilon,b)=\frac{\varepsilon}{16L^{-1}b},
\]
then we get
\begin{align*}
\phi(x,y)&\leq 4L^{-1}b^2\rho_X(4\norm{x-y}/b)\\
&\leq 4L^{-1}b4\norm{x-y}\\
&=16L^{-1}b\norm{x-y}\\
&\leq\varepsilon.
\end{align*}
For another, if
\[
\phi(x,y)\leq\tilde{\lambda}(\varepsilon,b)=\frac{b^2}{L}\eta\left(\frac{\varepsilon}{4b}\right),
\]
then we get
\begin{align*}
2L^{-1}b^2\eta(\norm{x-y}/4b)&\leq 2L^{-1}b^2\delta_X(\norm{x-y}/4b)\\
&\leq\phi(x,y)\\
&<\frac{2b^2}{L}\eta\left(\frac{\varepsilon}{4b}\right).
\end{align*}
Thus, in particular, we have $\eta(\norm{x-y}/4b)<\eta(\varepsilon/4b)$ and thus $\norm{x-y}/4b<\varepsilon/4b$, i.e.\ $\norm{x-y}\leq\varepsilon$, as $\eta$ is nondecreasing.
\end{proof}

As discussed before, from $\lambda$ and $\Lambda$ we can immediately derive 
\[
\theta(k,b)=\ceil*{\frac{L}{b^2}\left(\eta\left(\frac{1}{8b(k+1)}\right)\right)^{-1}}\dotdiv 1
\]
as a modulus for the triangularity and
\[
\widetilde{\theta}(k,b)=\ceil*{\frac{L}{b^2}\left(\eta\left(\frac{1}{8b(\ceil*{\frac{k+1}{16L^{-1}b}}+1)}\right)\right)^{-1}}\dotdiv 1
\]
as a modulus for the weak triangularity of $\phi$.

\section{Moduli of regularity and rates of convergence}

To establish the main results for rates of convergence under a metric regularity assumption, we now follow the setup established in \cite{KLAN2019}. Between this and the previous, there are some notational conflicts but the context will make it clear to which objects we refer. 

\medskip

The first main difference is that while the previous setup operated on approximations $AF_k$ of a solution set $F$ governed by natural numbers $k$, the setup from \cite{KLAN2019} works with a general function $F:X\to\overline{\mathbb{R}}$, mapping into the extended reals, for measuring the ``approximateness'' of any point $x$ to the solution set now represented by $\mathrm{zer}F$, i.e.\ instead of finding points in a set $F$ we are now interested in finding zeros of the function $F$ and $\vert F(x)\vert <\varepsilon$ is to be understood as $x$ being an $\varepsilon$-approximate solution.

\medskip

The whole setup of Fej\'er monotonicity is naturally formulated in this framework with Fej\'er monotonicity now understood over $\mathrm{zer}F$ and with approximate solutions understood in the previously explained sense instead of the sets $AF_k$.

\medskip

The main notion on which we rely in the following is that of a modulus of regularity from \cite{KLAN2019}:

\begin{definition}[essentially \cite{KLAN2019}]
Let $F:X\to\overline{\mathbb{R}}$ be a given function. For a set $S\subseteq X$, a function $\rho:\mathbb{R}^*_+\to\mathbb{R}^*_+$ is called a modulus of regularity for $F$ w.r.t.\ $\mathrm{zer}F$ and $S$ if for all $\varepsilon>0$ and all $x\in S$:
\[
\vert F(x)\vert<\rho(\varepsilon)\rightarrow\mathrm{dist}(x,\mathrm{zer}F)<\varepsilon.
\]
\end{definition}

In that way, a modulus of regularity, given an error $\varepsilon>0$, allows for the construction of an error $\rho(\varepsilon)$ such that being a $\rho(\varepsilon)$-approximate solution implies that one is $\varepsilon$-close to an actual solution. We refer to \cite{KLAN2019} for a discussion on how various notions of regularity known from optimization and nonlinear analysis fit into this framework and do not discuss this here any further.

\begin{remark}
Note that in \cite{KLN2018}, the above notion of a modulus of regularity is only introduced relative to the specific choice of $S=\overline{B}(z,r)$ for a given $z\in\mathrm{zer}F$ and $r>0$ but we here allow general sets $S$ (mainly out of simplicity so that the comparisons to moduli of regularity relative to a distance function $\phi$ introduced later become simpler and do not require tedious translations between metric balls and balls relative to $\phi$). In particular, we want to note that in the context of metrically bounded sets $S$, the uniformity of the modulus $\rho$ regarding the elements $x\in S$ is in particular suggested by the logical metatheorems of proof mining. Although this is not the case for general sets $S$, we here nevertheless included this uniformity in the above definition of the moduli of regularity. Note that in all case studies presented later, the sets are naturally bounded anyhow.
\end{remark}

In the context of the mapping $\phi$, we here consider the following relativized notion:

\begin{definition}
Let $\phi:X\times X\to\mathbb{R}_+$ be any mapping. For a set $S\subseteq X$, a function $\rho:\mathbb{R}^*_+\to\mathbb{R}^*_+$ is called a modulus of $\phi$-regularity for $F$ w.r.t.\ $\mathrm{zer}F$ and $S$ if for all $\varepsilon>0$ and all $x\in S$:
\[
\vert F(x)\vert<\rho(\varepsilon)\rightarrow\mathrm{dist}_\phi(x,\mathrm{zer}F)<\varepsilon.
\]
Here, $\mathrm{dist}_\phi$ is the $\phi$-distance function defined by 
\[
\mathrm{dist}_\phi(x,\mathrm{zer}F)=\inf_{y\in\mathrm{zer}F}\phi(y,x).
\]
\end{definition}

We can then obtain the following generalization of the main result of \cite{KLAN2019} on rates of convergence derived from moduli of regularity. We state the result already in its full generality. To that end, the sequence is only assumed to be partially $(\phi_n)$-$(G,H)$-quasi-Fej\'er monotone w.r.t.\ $F$, $f$ and $(\varepsilon_n)$ in the sense of before. In that way, it also generalizes the result given in \cite{Pis2022} which extended Theorem 4.1 of \cite{KLAN2019} by error terms to treat quasi-Fej\'er monotone sequences.

\medskip

In this setup of working with real $\varepsilon$'s as representations for errors, the previous moduli $\alpha_G$, $\beta_H$ as well as $\pi$ and $\theta$ then need to be real-valued in the sense that now:
\begin{enumerate}
\item A modulus of uniform weak triangularity or of uniform triangularity is now a function $\theta:\mathbb{R}^*_+\to\mathbb{R}^*_+$ that satisfies 
\[
\phi(y,x),\phi(y,z)\leq\theta(\varepsilon)\to \phi(x,z)\leq\varepsilon\text{ or }\phi(y,x),\phi(y,z)\leq\theta(\varepsilon)\to d(x,z)\leq\varepsilon
\]
for any $\varepsilon\in\mathbb{R}^*_+$ and any $x,y,z\in X$, respectively.
\item $G$- and $H$-moduli are now functions $\alpha_G,\beta_H:\mathbb{R}^*_+\to\mathbb{R}^*_+$ that satisfy
\[
a\leq\alpha_G(\varepsilon)\rightarrow G(a)\leq\varepsilon
\]
as well as
\[
H(a)\leq\beta_H(\varepsilon)\rightarrow a\leq\varepsilon
\]
for any $\varepsilon\in\mathbb{R}^*_+$ and any $a\in\mathbb{R}_+$, respectively.
\item A rate for $\phi_n\nearrow\phi$ uniformly is now a function $\pi:\mathbb{R}_+^*\to\mathbb{N}$ that satisfies
\[
\forall n\geq \pi(\varepsilon)\forall x,y \in X\left( \phi_n(x,y)\leq \phi(x,y)+\varepsilon\right)
\]
for all $\varepsilon\in\mathbb{R}^*_+$.
\end{enumerate}

Also, we now understand an $A$ witnessing $\psi\preceq^A\phi$ to be a real-valued function $A:\mathbb{R}^*_+\to \mathbb{R}^*_+$ with
\[
\psi\preceq^A\phi\equiv \forall\varepsilon>0\forall x,y\in X\left( \phi(x,y)\leq A(\varepsilon)\to \psi(x,y)\leq\varepsilon\right).
\]

\begin{theorem}\label{thm:modregthm}
Let $(X,d)$ be a metric space and $(\phi_n)$ be a sequence of functions $\phi_n:X\times X\to\mathbb{R}_+$ such that $\psi\preceq^A\phi_n$ for $A:\mathbb{R}^*_+\to \mathbb{R}^*_+$ and where $\psi:X\times X\to\mathbb{R}_+$ is uniformly weakly triangular with modulus $\theta$ and such that $\phi_n\nearrow\phi$ uniformly with a rate $\pi$ for a distance $\phi:X\times X\to\mathbb{R}_+$. Let $F:X\to\overline{\mathbb{R}}$ be given with $\mathrm{zer}F\neq\emptyset$. Let $f$ be a nondecreasing function with $f(n)\leq n$ and let $\kappa$ be a rate of divergence for $f$. Assume that $(x_n)$ is partially $(\phi_n)$-$(G,H)$-quasi-Fej\'er monotone w.r.t.\ $F$, $f$ and $(\varepsilon_n)$. Let $\alpha_G$ be a G-modulus for $G$ and $\beta_H$ be a H-modulus for $H$. Suppose there is a $\tau:\mathbb{R}^*_+\times \mathbb{N}\to\mathbb{N}$ such that
\[
\forall\delta >0\forall n\in\mathbb{N}\exists N\in [n;\tau(\delta,n)](\vert F(x_{2f(N)})\vert <\delta).
\]
Let $\rho$ be a modulus of $\phi$-regularity for $F$ w.r.t.\ $\mathrm{zer}F$ and $S$ such that $(x_n)\subseteq S$ and let $\xi$ be a Cauchy rate for $\sum\varepsilon_{n}<\infty$.

Then $(x_n)$ is $\psi$-Cauchy and
\[
\forall \delta>0\forall n,m\geq 2\mu(\delta)\;(\psi(x_n,x_m)<\delta)
\]
with 
\[
\mu(\delta)=\tau\left(\rho\left(\alpha_G\left(\frac{\beta_H(A(\theta(\delta))
)}{2}\right)/2\right),\max\left\{\kappa\left(\xi\left(\frac{\beta_H(A(\theta(\delta)))}{2}\right)\right),\kappa\left(\pi\left(\alpha_G\left(\frac{\beta_H(A(\theta(\delta)))}{2}\right)/2\right)\right)\right\}\right).
\]
If $\psi$ is uniformly triangular with a modulus $\theta$, then $(x_n)$ is Cauchy with the same rate.
\end{theorem}
\begin{proof}
Let $\delta>0$ be given. By assumption we have
\begin{gather*}
\exists N\in \left[\max\left\{\kappa\left(\xi\left(\frac{\beta_H(A(\theta(\delta)))}{2}\right)\right),\kappa\left(\pi\left(\alpha_G\left(\frac{\beta_H(A(\theta(\delta)))}{2}\right)/2\right)\right)\right\};\mu(\delta)\right]\\
\left(\vert F(x_{2f(N)})\vert <\rho\left(\alpha_G\left(\frac{\beta_H(A(\theta(\delta)))}{2}\right)/2\right)\right).
\end{gather*}
As $\rho$ is a corresponding modulus of regularity and since $(x_n)\subseteq S$, we get
\[
\mathrm{dist}_\phi(x_{2f(N)},\mathrm{zer} F)<\alpha_G\left(\frac{\beta_H(A(\theta(\delta)))}{2}\right)/2
\]
and therefore, there exists a $y\in\mathrm{zer} F$ with $\phi(y,x_{2f(N)})<\alpha_G\left(\frac{\beta_H(A(\theta(\delta)))}{2}\right)/2$. As
\[
N\geq \kappa\left(\pi\left(\alpha_G\left(\frac{\beta_H(A(\theta(\delta)))}{2}\right)/2\right)\right)
\] 
we obtain 
\[
f(N)\geq \pi\left(\alpha_G\left(\frac{\beta_H(A(\theta(\delta)))}{2}\right)/2\right)
\]
and thus we get $\phi_{2f(N)}(y,x_{2f(N)})\leq \alpha_G\left(\frac{\beta_H(A(\theta(\delta)))}{2}\right)$ and thus $G(\phi_{2f(N)}(y,x_{2f(N)}))\leq\frac{\beta_H(A(\theta(\delta)))}{2}$. As further
\[
N\geq\kappa\left(\xi\left(\frac{\beta_H(A(\theta(\delta)))}{2}\right)\right),
\]
we obtain
\[
N\geq f(N)\geq\xi\left(\frac{\beta_H(A(\theta(\delta)))}{2}\right)
\]
and thus for $n\geq 2N$, by $(\phi_{2n})$-$(G,H)$-quasi-Fej\'er monotonicity, we obtain
\begin{align*}
H(\phi_n(y,x_n))&\leq G(\phi_{2f(N)}(y,x_{2f(N)}))+\sum_{i=f(N)}^\infty\varepsilon_{i}\\
&\leq \beta_H(A(\theta(\delta)))/2 + \beta_H(A(\theta(\delta)))/2\\
&\leq\beta_H(A(\theta(\delta)))
\end{align*}
for $n\geq 2N$ even and by $(\phi_n)$-$(G,H)$-quasi-$f$-monotonicity, we obtain similarly
\[
H(\phi_n(y,x_n))\leq G(\phi_{2f(N)}(y,x_{2f(N)}))+\sum_{i=f(N)}^\infty\varepsilon_{i}\leq\beta_H(A(\theta(\delta)))
\]
for $n\geq 2N$ odd. Thus $\phi_n(y,x_n)\leq A(\theta(\delta))$ for any $n\geq 2N$. This yields $\psi(y,x_n)\leq\theta(\delta)$ for any $n\geq 2N$ (and so for any $n\geq 2\mu(\delta)$ since $N\leq \mu(\delta)$). If $\psi$ is uniformly weakly triangular with modulus $\theta$, we in particular derive from this that $\psi(x_n,x_m)\leq\delta$ for any $n,m\geq 2\mu(\delta)$. If $\psi$ is even uniformly triangular with modulus $\theta$, then we derive from this that $d(x_n,x_m)\leq\delta$ for any $n,m\geq 2\mu(\delta)$.
\end{proof}

\begin{remark}\label{rem:noComp}
Note that in the context of the above Theorem \ref{thm:modregthm}, we in particular get that the sequence in question strongly converges without any additional compactness assumption.
\end{remark}

We again explicitly list the case where all $\phi_n$ coincide with a single function $\phi$:

\begin{theorem}\label{thm:modregthm}
Let $(X,d)$ be a metric space and $\phi:X\times X\to\mathbb{R}_+$ be a function such that $\psi\preceq^A\phi$ for $A:\mathbb{R}^*_+\to \mathbb{R}^*_+$ and for $\psi:X\times X\to\mathbb{R}_+$ being uniformly weakly triangular with modulus $\theta$. Let $F:X\to\overline{\mathbb{R}}$ be given with $\mathrm{zer}F\neq\emptyset$. Let $f$ be a nondecreasing function with $f(n)\leq n$ and let $\kappa$ be a rate of divergence for $f$. Assume that $(x_n)$ is partially $(\phi_n)$-$(G,H)$-quasi-Fej\'er monotone w.r.t.\ $F$, $f$ and $(\varepsilon_n)$. Let $\alpha_G$ be a G-modulus for $G$, $\beta_H$ be an $H$-modulus for $H$. Suppose there is a $\tau:\mathbb{R}^*_+\times\mathbb{N}\to\mathbb{N}$ such that
\[
\forall\delta >0\forall n\in\mathbb{N}\exists N\in [n;\tau(\delta,n)](\vert F(x_{2f(N)})\vert <\delta).
\]
Let $\rho$ be a modulus of $\phi$-regularity for $F$ w.r.t.\ $\mathrm{zer}F$ and $S$ such that $(x_n)\subseteq S$ and let $\xi$ be a Cauchy rate for $\sum\varepsilon_{n}<\infty$.
Then $(x_n)$ is $\psi$-Cauchy and
\[
\forall \delta>0\forall n,m\geq 2\mu(\delta)\;(\psi(x_n,x_m)<\delta)
\]
with 
\[
\mu(\delta)=\tau\left(\rho\left(\alpha_G\left(\frac{\beta_H(A(\theta(\delta)))}{2}\right)/2\right),\kappa\left(\xi\left(\frac{\beta_H(A(\theta(\delta)))}{2}\right)\right)\right).
\]
If $\psi$ is uniformly triangular with a modulus $\theta$, then $(x_n)$ is Cauchy with the same rate.

Lastly, if $\varepsilon_n=0$ for any $n\in\mathbb{N}$, then the assumption that $f(n)\to\infty$ for $n\to\infty$ can be omitted and it suffices to assume that $\tau:\mathbb{R}^*_+\to\mathbb{N}$ is such that
\[
\forall\delta >0\exists N\leq \tau(\delta)(\vert F(x_{2f(N)})\vert <\delta),
\]
to derive that $(x_n)$ is $\psi$-Cauchy or Cauchy, depending on the assumption on $\psi$, with a rate $2\mu'(\delta)$ where
\[
\mu'(\delta)=\tau\left(\rho\left(\alpha_G\left(\frac{\beta_H(A(\theta(\delta)))}{2}\right)/2\right)\right)
\]
and where the other constants are as before.
\end{theorem}

We omit the proof as it is rather immediate.

\medskip

Again, if we have moduli relating the distance function $\phi$ to the usual metric distance in the sense that we have moduli $\lambda,\Lambda:\mathbb{R}^*_+\times\mathbb{N}\to\mathbb{R}^*_+$ such that
\begin{gather*}
\phi(x,y)\leq\lambda(\varepsilon,b)\to d(x,y)\leq\varepsilon,\\
d(x,y)\leq\Lambda(\varepsilon,b)\to \phi(x,y)\leq\varepsilon,
\end{gather*}
for any $\varepsilon>0$ and any $x,y\in X$ such that $d(a,x),d(a,y)\leq b$, then we can modify the above result to only require the usual standard metric versions of the assumptions. Similarly, we also in this case call $\phi$ consistent with moduli $\lambda,\Lambda:\mathbb{R}^*_+\times\mathbb{N}\to\mathbb{R}^*_+$ if these satisfy the above inequalities and we call $\phi$ uniformly consistent with moduli $\lambda,\Lambda:\mathbb{R}^*_+\to\mathbb{N}$ if these satisfy the above inequalities without any reference to $b$.

\medskip

As before, we can again immediately derive a corresponding moduli $\theta$ and $\widetilde{\theta}$ from the moduli $\lambda,\Lambda$ for the triangularity and weak triangularity of $\phi$, respectively: Define
\[
\theta(\varepsilon,b)=\lambda\left(\frac{\varepsilon}{2},b\right).
\]
Then for $x,y,z\in X$ such that $d(a,x),d(a,y),d(a,z)\leq b$, assume that we have $\phi(y,x),\phi(y,z)\leq \theta(\varepsilon,b)$. Then in particular $d(y,x),d(y,z)\leq \varepsilon/2$ by assumption on $\lambda$ and thus $d(x,z)\leq \varepsilon$. Thus $\theta$ is a modulus for the triangularity of $\theta$. It is then clear that
\[
\widetilde{\theta}(\varepsilon,b)=\lambda\left(\frac{\Lambda(\varepsilon,b)}{2},b\right)
\]
is a corresponding modulus for the weak triangularity.

\medskip

Similarly, we can provide a translation for the main notion of a modulus of $\phi$-regularity to its metric counterpart.
\begin{lemma}
Let $\phi:X\times X\to\mathbb{R}_+$ be a uniformly consistent mapping with moduli $\lambda,\Lambda:\mathbb{R}^*_+\to\mathbb{R}^*_+$. Let further $z\in\mathrm{zer}F$, $r>0$ and $\rho:\mathbb{R}^*_+\to\mathbb{R}^*_+$ be a modulus of regularity for $F$ w.r.t.\ $\mathrm{zer}F$ and $S$. Then $\rho'(\varepsilon)=\rho(\Lambda(\varepsilon/2))$ is a modulus of $\phi$-regularity for $F$ w.r.t.\ $\mathrm{zer}F$ and $S$.
\end{lemma}
\begin{proof}
Let $\varepsilon$ be given and let $x\in S$. Then suppose that $\vert F(x)\vert<\rho'(\varepsilon)=\rho(\Lambda(\varepsilon/2))$, then by assumption on $\rho$, we have $\mathrm{dist}(x,\mathrm{zer}F)<\Lambda(\varepsilon/2)$ and thus
\[
\exists y\in\mathrm{zer}F\left( d(y,x)<\Lambda(\varepsilon/2)\right).
\]
Thus for this $y$, we have $\phi(y,x)\leq\varepsilon/2<\varepsilon$ and thus it follows that $\mathrm{dist}_\phi(x,\mathrm{zer}F)<\varepsilon$.
\end{proof}

In the following theorem, we now just collect the resulting rates of convergence that can be constructed in the context of a consistent distance:

\begin{theorem}\label{thm:modregthmetricversion}
Let $(X,d)$ be a metric space and $(\phi_n)$ be a sequence of functions $\phi_n:X\times X\to\mathbb{R}_+$ such that $\psi\preceq^A\phi_n$ for $A:\mathbb{R}^*_+\to\mathbb{R}^*_+$ and where $\psi:X\times X\to\mathbb{R}_+$ is uniformly weakly triangular with a modulus $\theta$. Further assume that $\phi_n\nearrow\phi$ uniformly with a rate $\pi$ for a given function $\phi:X\times X\to\mathbb{R}_+$ which is uniformly consistent with moduli $\lambda,\Lambda:\mathbb{R}^*_+\to\mathbb{R}^*_+$. Let $F:X\to\overline{\mathbb{R}}$ be given with $\mathrm{zer}F\neq\emptyset$. Let $f$ be a nondecreasing function with $f(n)\leq n$ and let $\kappa$ be a rate of divergence for $f$. Assume that $(x_n)$ is partially $(\phi_n)$-$(G,H)$-quasi-Fej\'er monotone w.r.t.\ $F$, $f$ and $(\varepsilon_n)$. Let $\alpha_G$ be a G-modulus for $G$, $\beta_H$ be an $H$-modulus for $H$. Suppose there is a $\tau:\mathbb{R}^*_+\times\mathbb{N}\to\mathbb{N}$ such that
\[
\forall\delta >0\forall n\in\mathbb{N}\exists N\in [n;\tau(\delta,n)](\vert F(x_{2f(N)})\vert <\delta).
\]
Let $\rho$ be a modulus of regularity for $F$ w.r.t.\ $\mathrm{zer}F$ and $S$ such that $(x_n)\subseteq S$ and let $\xi$ be a Cauchy rate for $\sum\varepsilon_{n}<\infty$.

Then $(x_n)$ is $\psi$-Cauchy and
\[
\forall \delta>0\forall n,m\geq 2\mu'(\delta)\; (\psi(x_n,x_m)<\delta)
\]
with $\mu'$ defined as $\mu$ in Theorem \ref{thm:modregthm} but with $\rho'(\varepsilon)=\rho(\Lambda(\varepsilon/2))$ instead of $\rho$. If $\psi$ is even uniformly triangular with a modulus $\theta$, then the sequence is Cauchy with the same rate.

If $\phi_n=\phi$ and $\varepsilon_n=0$ for all $n$, then the assumption that $f(n)\to\infty$ for $n\to\infty$ can be omitted and we may assume that $\tau:\mathbb{R}^*_+\to\mathbb{N}$ is such that
\[
\forall\delta >0\exists N\leq \tau(\delta)(\vert F(x_{2f(N)})\vert <\delta),
\]
to derive that $(x_n)$ is $\psi$-Cauchy or Cauchy, depending on whether $\psi$ is uniformly weakly triangular or uniformly triangular with modulus $\theta$ as before, with a rate $2\mu'(\delta)$ where
\[
\mu'(\delta)=\tau\left(\rho'\left(\alpha_G\left(\frac{\beta_H(A(\theta(\delta)))}{2}\right)/2\right)\right)
\]
and where the other constants are as before.
\end{theorem}

Again, we omit the proof as it is rather immediate.

\section{Applications to a method with alternating inertia}

Let $X$ be a finite dimensional Hilbert space with inner product $\langle\cdot,\cdot\rangle$ and induced norm $\norm{\cdot}$ and dimension $d$. We consider the algorithm given in \cite{IH2019} to approximate fixed points of $\alpha$-averaged mappings $T$, i.e.\ $T:X\to X$ satisfies
\[
(1-\alpha)\norm{(Id-T)x-(Id-T)y}^2\leq\alpha\left(\norm{x-y}^2-\norm{T x-T y}^2\right).
\]
Concretely, given $x^0$, $\lambda>0$ and $0\leq\alpha_k\leq(1-\alpha)/\alpha$ for all $k$, we consider the iteration defined by
\[
x^{k+1}=T\bar{x}^k\text{ with }\bar{x}^k=\begin{cases}x^k&\text{if }k\text{ is even},\\
x^k+\alpha_k(x^k-x^{k-1})&\text{if }k\text{ is odd}.\end{cases}
\]
By modifying a result from \cite{MP2015} for the case of $\alpha=1/2$, the authors then obtain the following convergence result:
\begin{theorem}[\cite{IH2019}, Lemma 3.3]
Let $0\leq\alpha_k\leq(1-\alpha)/\alpha$ for all $k$. Then $(x^k)$ converges to a fixed point of $T$.
\end{theorem}
The proof of this result proceeds by establishing partial Fej\'er monotonicity (w.r.t.\ the usual metric distance) and thus fits into the above general result which we will apply in the following to derive quantitative information.

\medskip

Note before that, that using the correspondence between averaged mappings and monotone operators (see, e.g., \cite{BMW2020} as well as \cite{BC2017}) this immediately covers the case when $T$ is given as the resolvent $J^A_\lambda=(Id+\gamma A)^{-1}$ for $\gamma>0$ for a maximal monotone or maximally $\rho$-comonotone operator $A:X\to 2^X$ as $J^A_\lambda$ is firmly nonexpansive, i.e.\ $1/2$-averaged, in the former case and $\alpha$-averaged for $\alpha=\frac{1}{2(\rho/\lambda+1)}$ in the latter case (see \cite{BMW2020}).

In that case, by $\mathrm{Fix}J^A_\lambda=A^{-1}(0)$, approximating fixed points of $T=J^A_\lambda$ means approximating zeros $\hat{x}\in A^{-1}(0)$.

In that formulation, as mentioned before, the algorithm considered above is exactly the algorithm considered in \cite{MP2015} for maximally monotone operators and the following quantitative analysis thus also applies in that case by setting $\alpha=1/2$.

\medskip
 
Now, for a quantitative approach in the spirit of the previous sections, take $\hat{x}$ with $\hat{x}=T\hat{x}$ fixed, suppose $M\geq\norm{x^k-\hat{x}},\norm{\bar{x}^k-\hat{x}}$ for all $k$. Define $X_0=\overline{B}(\hat{x},M)$ as well as
\[
F=\{x\in X_0\mid x=T x\}
\]
as well as
\[
AF_k=\left\{x\in X_0\mid \norm{x-T x}\leq\frac{1}{k+1}\right\}.
\]
\begin{lemma}
If $\sum_i a_i\leq A$, then
\[
\min_{0\leq i\leq k-1}a_i\leq\frac{A}{k}\text{ for all }k\geq 1.
\]
\end{lemma}
\begin{proof}
Suppose not. Then there is a $k\geq 1$ such that for any $0\leq i\leq k-1$: $a_i>A/k$. Then
\[
\sum_i a_i\geq\sum_{i=0}^{k-1}a_i>A
\]
which is a contradiction.
\end{proof}
\begin{lemma}[\cite{IH2019}]
For any $k$:
\[
\norm{x^{k+2}-\hat{x}}^2\leq\norm{x^k-\hat{x}}^2-\left(\frac{1-\alpha}{\alpha}\right)\norm{T x^k+\alpha_{k+1}(T x^k-x^k)-x^{k+2}}^2.
\]
\end{lemma}
\begin{lemma}\label{lem:MuPeng16Av}
\[
\left(\frac{1-\alpha}{\alpha}\right)\sum_{i}\norm{x^{2i+2}-\bar{x}^{2i+1}}^2=\left(\frac{1-\alpha}{\alpha}\right)\sum_{i}\norm{T x^{2i}-\alpha_{2i+1}(T x^{2i}-x^{2i})-x^{2i+2}}^2\leq\norm{x^0-\hat{x}}^2.
\]
Thus in particular
\[
\min_{0\leq i\leq k-1}\norm{x^{2i+2}-\bar{x}^{2i+1}}^2\leq\left(\frac{\alpha}{1-\alpha}\right)\frac{\norm{x^0-\hat{x}}^2}{k}
\]
for any $k\geq 1$.
\end{lemma}
We omit the proof as it is immediate.
\begin{lemma}\label{lem:MuPengLemma2Av}
For any $k\geq 0$: $\norm{x^{2k+3}-x^{2k+2}}\leq\norm{x^{2k+2}-\bar{x}^{2k+1}}$.
\end{lemma}
\begin{proof}
As $T$ is $\alpha$-averaged:
\[
(1-\alpha)\norm{x^{2k+2}-x^{2k+3}-(\bar{x}^{2k+1}-x^{2k+2})}^2\leq\alpha\left(\norm{x^{2k+2}-\bar{x}^{2k+1}}^2-\norm{x^{2k+3}-x^{2k+2}}^2\right).
\]
Thus, in particular
\begin{align*}
\norm{x^{2k+3}-x^{2k+2}}^2&\leq \norm{x^{2k+2}-\bar{x}^{2k+1}}^2-\left(\frac{1-\alpha}{\alpha}\right)\norm{x^{2k+2}-x^{2k+3}-(\bar{x}^{2k+1}-x^{2k+2})}^2\\
&\leq\norm{x^{2k+2}-\bar{x}^{2k+1}}^2.
\end{align*}
\end{proof}
\begin{lemma}\label{lem:approxFPointBound}
Let $b\geq\norm{x^0-\hat{x}}$. Then $\Phi(k)=2\max\left\{1,\ceil*{\frac{\alpha}{1-\alpha}b^2(k+1)^2}\right\}$ is an approximate $F$-point bound for $(x^{2n})$.
\end{lemma}
\begin{proof}
Write $k'$ for $\max\left\{1,\ceil*{\frac{\alpha}{1-\alpha}b^2(k+1)^2}\right\}$. As $k'\geq 1$, by Lemma \ref{lem:MuPeng16Av}, there exists an $i\in [0;k'-1]$ such that
\[
\norm{x^{2i+2}-\bar{x}^{2i+2}}^2\leq\left(\frac{\alpha}{1-\alpha}\right)\frac{\norm{x^0-\hat{x}}^2}{k'}\leq\left(\frac{\alpha}{1-\alpha}\right)\frac{b^2}{k'}.
\]
Using Lemma \ref{lem:MuPengLemma2Av}, we get
\[
\norm{x^{2i+3}-x^{2i+2}}\leq\left(\frac{\sqrt{\alpha}}{\sqrt{1-\alpha}}\right)\frac{b}{\sqrt{k'}}\leq\frac{1}{k+1}.
\]
As $2i+2$ is even, we have $x^{2i+3}=T\bar{x}^{2i+2}=T x^{2i+2}$. Thus $x^{2i+2}\in AF_k$. Lastly, we have $2i+2\leq 2(k'-1)+2=2k'=\Phi(k)$.
\end{proof}
\begin{lemma}\label{lem:MuPeng1819Av}
Let $\norm{x^*-T x^*}\leq\varepsilon$ and $b\geq\norm{x^k-x^*},\norm{x^k-T x^*},\norm{T x^k-x^*},\norm{x^k-T x^k}$ for any $k$. Then for any $k$:
\[
\norm{x^{k+2}-x^*}^2\leq\norm{x^k-x^*}^2+\left(\frac{3-\alpha}{\alpha^2}+2\right)b\varepsilon.
\]
\end{lemma}
\begin{proof}
We have
\begin{align*}
\norm{x^{k+2}-x^*}&=\langle x^{k+2}-T x^*,x^{k+2}-x^*\rangle+\langle T x^*-x^*,x^{k+2}-x^*\rangle\\
&=\langle x^{k+2}-T x^*,x^{k+2}-T x^*\rangle+\langle x^{k+2}-T x^*,T x^*-x^*\rangle+\langle T x^*-x^*,x^{k+2}-x^*\rangle\\
&\leq\norm{x^{k+2}-T x^*}^2+2b\varepsilon\\
&\leq\norm{T x^k+\alpha_{k+1}(T x^k-x^k)-x^*}^2+2b\varepsilon.
\end{align*}
Now
\begin{align*}
&\norm{T x^k+\alpha_{k+1}(T x^k-x^k)-x^*}^2\\
&\qquad\qquad=\norm{(1+\alpha_{k+1})(T x^k-x^*)-\alpha_{k+1}(x^{k}-x^*)}^2\\
&\qquad\qquad=(1+\alpha_{k+1})\norm{T x^k-x^*}^2-\alpha_{k+1}\norm{x^k-x^*}^2+\alpha_{k+1}(1+\alpha_{k+1})\norm{T x^k-x^k}^2
\end{align*}
using Corollary 2.15 from \cite{BC2017}. Using
\begin{align*}
\norm{T x^k-x^*}^2&=\langle T x^k-T x^*,T x^k-x^*\rangle+\langle T x^*-x^*,T x^k-x^*\rangle\\
&\leq\langle T x^k-T x^*,T x^k-T x^*\rangle +2b\varepsilon\\
&\leq\norm{x^k-x^*}^2-\frac{1-\alpha}{\alpha}\norm{x^k-T x^k-x^*+T x^*}^2+2b\varepsilon,
\end{align*}
we get
\begin{align*}
\norm{T x^k+\alpha_{k+1}(T x^k-x^k)-x^*}^2&\leq \norm{x^k-x^*}^2-(1+\alpha_{k+1})\frac{1-\alpha}{\alpha}\norm{x^k-T x^k-x^*+T x^*}^2\\
&\qquad\qquad\qquad+\alpha_{k+1}(1+\alpha_{k+1})\norm{T x^k-x^k}^2+(1+\alpha_{k+1})2b\varepsilon.
\end{align*}
Using
\begin{align*}
\norm{x^k-T x^k-x^*+T x^*}^2&=\norm{x^k-T x^k}^2+ \langle x^k -T x^k, T x^*-x^*\rangle+\langle T x^*-x^*, x^k-T x^k-x^*+T x^*\rangle\\
&\geq\norm{x^k-T x^k}^2 -b\varepsilon+\langle T x^*-x^*, x^k-x^*\rangle+\langle T x^*-x^*, T x^*-T x^k\rangle\\
&\geq \norm{x^k-T x^k}^2 -3b\varepsilon
\end{align*}
this implies
\begin{align*}
\norm{T x^k+\alpha_{k+1}(T x^k-x^k)-x^*}^2&\leq \norm{x^k-x^*}^2-(1+\alpha_{k+1})\frac{1-\alpha}{\alpha}\left(\norm{x^k-T x^k}^2 -3b\varepsilon\right)\\
&\qquad\qquad\qquad+\alpha_{k+1}(1+\alpha_{k+1})\norm{T x^k-x^k}^2+(1+\alpha_{k+1})2b\varepsilon\\
&\leq \norm{x^k-x^*}^2+(1+\alpha_{k+1})\frac{1-\alpha}{\alpha}3b\varepsilon+(1+\alpha_{k+1})2b\varepsilon\\
&\leq \norm{x^k-x^*}^2+\left(1+\frac{1-\alpha}{\alpha}\right)\frac{1-\alpha}{\alpha}3b\varepsilon+\left(1+\frac{1-\alpha}{\alpha}\right)2b\varepsilon\\
&=\norm{x^k-x^*}^2+\left(\frac{3-\alpha}{\alpha^2}\right)b\varepsilon
\end{align*}
which yields
\[
\norm{x^{k+2}-x^*}^2\leq\norm{x^k-x^*}^2\left(\frac{3-\alpha}{\alpha^2}+2\right)b\varepsilon.
\]
\end{proof}
\begin{lemma}\label{lem:MuPengTheorem1iAv}
Let $b\geq\norm{x^k-x^*},\norm{x^k-T x^*},\norm{T x^k-x^*},\norm{x^k-T x^k}$ for all $k$ and $\norm{x^*-T x^*}\leq\varepsilon$. Then for any $k$:
\begin{enumerate}
\item $\norm{x^{2k+2}-x^*}^2\leq\norm{x^{2k}-x^*}^2+\left(\frac{3-\alpha}{\alpha^2}+2\right)b\varepsilon$,
\item $\norm{x^{2k+3}-x^*}^2\leq\norm{x^{2k}-x^*}^2+\left(\frac{3-\alpha}{\alpha^2}+4\right)b\varepsilon$.
\end{enumerate}
\end{lemma}
\begin{proof}
The first item follows immediately from Lemma \ref{lem:MuPeng1819Av} and for the latter, note that
\[
\norm{x^{2k+3}-x^*}^2\leq\norm{x^{2k+2}-x^*}+2b\varepsilon.
\]
\end{proof}
This immediately yields a modulus of uniform Fej\'er monotonicity and the respective uniform modulus for the odd elements:
\begin{lemma}\label{lem:modulusUnifFejer}
For all $r,n,m$:
\begin{enumerate}
\item $\forall x^*\in AF_{\chi(n,m,r)}\forall l\leq m\left( \norm{x^{2(n+l)}-x^*}\leq\norm{x^{2n}-x^*}+\frac{1}{r+1}\right)$,
\item $\forall x^*\in AF_{\zeta(n,m,r)}\forall l\leq m\left(\norm{x^{2(n+l)+3}-x^*}\leq\norm{x^{2n}-x^*}+\frac{1}{r+1}\right)$,
\end{enumerate}
where
\[
\chi(n,m,r)=2m^2\ceil*{\frac{3-\alpha}{\alpha^2}+2}M(r+1)^2\dotdiv 1\text{ and }\zeta(n,m,r)=2m^2\ceil*{\frac{3-\alpha}{\alpha^2}+4}M(r+1)^2\dotdiv 1.
\]
\end{lemma}
\begin{proof}
W.l.o.g.\  $l\geq 1$ in both cases. Let $x^*\in AF_{\chi(n,m,r)}$. Then, by Lemma \ref{lem:MuPengTheorem1iAv}, we get
\begin{align*}
\norm{x^{2(n+l)}-x^*}&\leq\norm{x^{2n}-x^*}+m\sqrt{2\ceil*{\frac{3-\alpha}{\alpha^2}+2}M\frac{1}{\chi(n,m,r)+1}}\\
&\leq\norm{x^{2n}-x^*}+\frac{1}{r+1}.
\end{align*}
Similarly, we obtain for $x^*\in AF_{\zeta(n,m,r)}$:
\begin{align*}
\norm{x^{2(n+l)+3}-x^*}&\leq\norm{x^{2n}-x^*}+m\sqrt{2\ceil*{\frac{3-\alpha}{\alpha^2}+4}M\frac{1}{\zeta(n,m,r)+1}}\\
&\leq\norm{x^{2n}-x^*}+\frac{1}{r+1}.
\end{align*}
\end{proof}
Combined, we obtain a rate of metastability for the sequences by an application of Theorem \ref{thm:metastabfversionmetricversion}.
\begin{theorem}
The sequence $(x_n)$ is Cauchy and for all $k\in\mathbb{N},g\in\mathbb{N}^\mathbb{N}$:
\[
\exists N\leq\Psi(k,g)\forall i,j\in[N;N+g(N)]\left( \norm{x_i-x_j}\leq\frac{1}{k+1}\text{ and }x_i\in AF_k\right)
\]
where $\Psi(k,g)=2\Psi_0(P,2k+1,g')+3$ with $P=\ceil*{(64k+64)\sqrt{d}M}^d$, $g'(n)=g(n+3)+3$ and $\Psi_0$ defined by
\[
\begin{cases}
\Psi_0(0,k,g)=0,\\
\Psi_0(n+1,k,g)=\Phi^M(\eta^M_{g,k}(\Psi_0(n,k,g),4k+3),
\end{cases}
\]
where
\[
\Phi(k)=2\max\left\{1,\ceil*{\frac{\alpha}{1-\alpha}M^2(k+1)^2}\right\}+1
\]
and $\Phi^M(k)=\max\{\Phi(j)\mid j\leq k\}$ as well as
\[
\eta_{g,k}(n,r)=\max\left\{2k+1,2\floor*{\frac{g(2n)}{2}}^2\ceil*{\frac{3-\alpha}{\alpha^2}+4}M(2r+4)^2\dotdiv 1\right\}.
\]
\end{theorem}
\begin{proof}
Note that $T$ is nonexpansive and thus by Lemma 7.1 of \cite{KLN2018}, we get hat $F$ is uniformly closed with moduli $\omega_F(k)=4k+3$ and $\delta_F(k)=2k+1$. Further, by Example 2.8 of \cite{KLN2018} we get that $\gamma(k)=\ceil*{2(k+1)\sqrt{d}M}^d$ is a modulus of total boundedness of $\overline{B}(\hat{x},M)$ Then, for the sequence $(\tilde{x}_n)$ defined by 
\[
\tilde{x}_n=\begin{cases}
x_n&\text{for }n\geq 3,\\
x_n&\text{for even }n< 3,\\
x_{n-1}&\text{for odd }n<3,
\end{cases}
\]
we get from Theorem \ref{thm:metastabfversionmetricversion} by the use of Proposition \ref{pro:affineLinearF}, Proposition \ref{pro:affineLinearFapprox}, Lemma \ref{lem:modulusUnifFejer}, Lemma \ref{lem:approxFPointBound} and $f(n):=n\dotdiv 1$ that 
\[
\exists N\leq2\Psi_0(P,2k+1,g')\forall i,j\in[N;N+g(N+3)+3]\left( \norm{\tilde{x}_i-\tilde{x}_j}\leq\frac{1}{k+1}\text{ and }\tilde{x}_i\in AF_k\right)
\]
after some obvious simplifications. Therefore we get, using $\tilde{x}_i=x_i$ for $i\geq 3$, that 
\[
\exists N\leq2\Psi_0(P,2k+1,g')\forall i,j\in[N+3;N+3+g(N+3)]\left( \norm{x_i-x_j}\leq\frac{1}{k+1}\text{ and }x_i\in AF_k\right)
\]
which yields the result.
\end{proof}

Under the assumption of a modulus of regularity, we can also derive a rate of convergence using Theorem \ref{thm:modregthmetricversion}. For this, the central instantiation of the notion of the moduli of regularity is the following:

\begin{definition}
Let $T:X\to X$ be a map. For a set $S\subseteq X$, a function $\rho:\mathbb{R}^*_+\to\mathbb{R}^*_+$ is called a modulus of regularity for $T$ w.r.t.\ $\mathrm{Fix}T$ and $S$ if for all $\varepsilon>0$ and all $x\in S$:
\[
\norm{x-Tx}<\rho(\varepsilon)\rightarrow\mathrm{dist}(x,\mathrm{Fix}T)<\varepsilon.
\]
\end{definition}

With such a modulus, we get the following result.

\begin{theorem}
Let $x^*\in\mathrm{zer}F$ and $b\geq\max\{\norm{x^0-x^*},\norm{x^1-x^*}\}$. Let $\rho$ be a modulus of regularity for $T$ w.r.t.\ $\mathrm{Fix}T$ and $\overline{B}(x^*,b)$. Then $(x_n)$ is Cauchy with
\[
\forall \delta>0\forall n,m\geq\max\left\{3,2\ceil*{\frac{\alpha}{1-\alpha}b^2\left(\ceil*{\frac{1}{\rho\left(\frac{\delta}{4}\right)}}+1\right)^2}\right\}(\norm{x_n-x_m}<\delta).
\]
\end{theorem}
\begin{proof}
For $(\tilde{x}_n)$ as before we get from Theorem \ref{thm:modregthmetricversion} (and using Proposition \ref{pro:affineLinearF}, Proposition \ref{pro:affineLinearFapprox}, Lemma \ref{lem:modulusUnifFejer}, Lemma \ref{lem:approxFPointBound} and $f(n):=n\dotdiv 1$) that
\[
\forall \delta>0\forall n,m\geq 2\ceil*{\frac{\alpha}{1-\alpha}b^2\left(\ceil*{\frac{1}{\rho\left(\frac{\delta}{4}\right)}}+1\right)^2}(\norm{\tilde{x}_n-\tilde{x}_m}<\delta).
\]
The result then follows as for $n\geq 3$, we have $\tilde{x}_n=x_n$.
\end{proof}

As already discussed in Remark \ref{rem:noComp}, also here the strong convergence in particular holds without any additional compactness assumptions in the context of these moduli of regularity.

\medskip

We refer to the extensive discussions in \cite{KLAN2019} for examples of mappings $T$ where such a modulus can be readily constructed.

\section{Applications to the Mann-type proximal point algorithm in Banach spaces}

For a second application, we work over a Banach space $X$ with norm $\norm{\cdot}$ and dual space $X^*$. Since the seminal work by Rockafellar \cite{Roc1976} and Martinet \cite{Mar1970}, the proximal point algorithm has been one of the most influential algorithmic approaches to many problems in optimization.

We are here concerned with the previously mentioned work \cite{KKT2004} for a combination of Mann's iteration with the proximal point algorithm (see again \cite{Xu2002} for this kind of iteration in Hilbert space and e.g.\ \cite{MT2004} for related iterations) for a monotone operator of the Banach space $X$ in the sense of Browder \cite{Bro1968}. Concretely, a set-valued operator $T\subseteq X\times X^*$ is called monotone if
\[
\langle x-y,x^*-y^*\rangle\geq 0
\]
for all $(x,x^*),(y,y^*)\in T$. For these monotone operators in Banach spaces, there are natural analogous for the resolvent. Over a smooth Banach space with a (single-valued) duality map $J$ and given a maximally monotone operator $T$, one particular choice is the function $J_r:X\to X$ defined by
\[
J_rx=(J+rT)^{-1}Jx
\]
for $r>0$ and $x\in X$ (see also \cite{Bar1976} for this and other variants of the resolvent in this context). Then, by combining the aspects of Mann's iteration and the proximal point point algorithm, Kamimura, Kohsaka and Takahashi obtained the following result:

\begin{theorem}[\cite{KKT2004}]
Let $X$ be a uniformly smooth and uniformly convex Banach space whose duality map is weakly sequentially continuous. Let $T$ be a maximally monotone operator and $x_n$ defined by
\[
x_{n+1}=J^{-1}\left(\alpha_n Jx_n+(1-\alpha_n)JJ_{r_n}x_n\right)
\]
for $x_0\in X$ and where $(\alpha_n)\subseteq [0,1]$, $(r_n)\subseteq \mathbb{R}^*_+$ satisfy
\[
\limsup_{n\to\infty}\alpha_n <1\text{ and }\liminf_{n\to\infty}r_n>0.
\]
If $T^{-1}0\neq\emptyset$, then the sequence $(x_n)$ converges weakly to an element of $T^{-1}0$.
\end{theorem}
Note that in this context of a uniformly smooth space, the mapping $J$ is uniformly continuous on bounded subsets.

\medskip

Essential for the proof of weak convergence, and thus also for the convergence analysis, is establishing Fej\'er monotonicity of this algorithm w.r.t.\ the previously discussed distance
\[
\phi(x,y)=\norm{x}^2-\langle x,Jy\rangle+\norm{y}^2
\]
for $x,y\in X$. In that way, this fits into the general quantitative result established before and we apply this result in the following to extract quantitative information on the above convergence from the proof.  Recall for this in particular the previously derived moduli $\lambda,\Lambda$ and $\theta$ for the distance $\phi$ in a uniformly convex and smooth space with a non-decreasing modulus of uniform convexity $\eta$:
\begin{gather*}
\Lambda(k,b)=\ceil*{\frac{k+1}{16L^{-1}b}}\dotdiv 1,\\
\lambda(k,b)=\ceil*{\frac{L}{b^2}\left(\eta\left(\frac{1}{4b(k+1)}\right)\right)^{-1}}\dotdiv 1,\\
\theta(k,b)=\ceil*{\frac{L}{b^2}\left(\eta\left(\frac{1}{8b(k+1)}\right)\right)^{-1}}\dotdiv 1.
\end{gather*}

For the quantitative analysis, we again suppose $M\geq\norm{x_n}$ for all $n$ and define $X_0=\overline{B}(0,M)$ as well as
\[
F=\mathrm{zer}T\cap X_0=\mathrm{Fix}J_{r}\cap X_0\; \text{(for any $r>0$)}
\]
as well as
\[
AF_k=\left\{x\in X_0\mid \norm{x-J_{r_n} x}\leq\frac{1}{k+1}\text{ for all }n\leq k\right\}.
\]

The result given in \cite{KKT2004} relies on a particular result on the relation between the distance and the resolvents of the operator from \cite{KT2004}:
\begin{lemma}[\cite{KT2004}]
Let $X$ be smooth, strictly convex and reflexive and let $T$ be maximally monotone with $\mathrm{zer}T\neq\emptyset$. Then
\[
\phi(x,J_ry)+\phi(J_ry,y)\leq\phi(x,y)
\]
for any $x\in\mathrm{zer}T$ and $y\in X$.
\end{lemma}
Based on the approximations $AF_k$, we get a quantitative version of this result. For this, we still need the following preliminary bound on the norm of the resolvent:
\begin{lemma}\label{lem:resMaj}
Let $T$ be maximally monotone and let $d\in Tc$ with $C\geq\norm{c}$, $D\geq\norm{d}$. Then for any $r>0$ and any $x\in X$:
\[
\norm{J_rx}\leq \max\{(1+C)(b+RD)+C,1\}=:\mu(R,b)
\]
where $R\geq r$ and $b\geq\norm{x}$.
\end{lemma}
\begin{proof}
By definition of the resolvent, we have for any $r>0$ and any $x\in X$:
\[
\left(J_r x, r^{-1}(Jx-JJ_r x)\right)\in T.
\]
As $d\in Tc$, we have by monotonicity of $T$ that
\[
\langle J_r x-c,r^{-1}(Jx-JJ_r x)-d\rangle\geq 0.
\]
By linearity, we obtain
\begin{align*}
0&\leq \langle J_r x-c,r^{-1}(Jx-JJ_r x)-d\rangle\\
&=\langle J_r x,r^{-1}Jx-d\rangle -\langle J_r x,r^{-1}JJ_r x\rangle -\langle c,r^{-1}Jx-d\rangle+\langle c,r^{-1}JJ_r x\rangle
\end{align*}
and therefore we get
\begin{align*}
\norm{J_r x}^2=\langle J_r x,JJ_r x\rangle&\leq \langle J_r x,Jx-r d\rangle -\langle c,Jx-r d\rangle+\langle c,JJ_r x\rangle\\
&\leq \norm{J_r x}(\norm{x}+r D)+C(\norm{x}+r D)+C\norm{J_r x}.
\end{align*}
Suppose now that $\norm{J_r x}>\max\{\norm{x}+r D+C(\norm{x}+r D)+C,1\}$. Then by multiplying by $\norm{J_r x}^{-1}$, we get 
\begin{align*}
\norm{J_r x}&\leq \norm{x}+r D+\norm{J_r x}^{-1}C(\norm{x}+r D)+C\\
&\leq \norm{x}+r D+C(\norm{x}+r D)+C
\end{align*}
from the above which is a contradiction. This gives the claim.
\end{proof}

\begin{lemma}\label{lem:Lem25quant}
Let $X$ be uniformly smooth and uniformly convex and let $T$ be maximally monotone. Let $d\in Tc$ with $C\geq\norm{c}$, $D\geq\norm{d}$. Let $\omega(\varepsilon,b)\leq\varepsilon$ be a modulus of uniform continuity of $J$ on bounded subsets, i.e.
\[
\forall z,w\in X,b\in\mathbb{N},\varepsilon>0\left(\norm{z},\norm{w}\leq b\land 
\norm{z-w}\leq\omega(\varepsilon,b)\to \norm{Jz-Jw}\leq\varepsilon\right).
\]
Then for any $\varepsilon>0$, $x,y\in X$ and $r,s>0$, if $x$ and $s$ are such that
\[
\norm{x-J_sx}\leq\omega\left(\frac{\varepsilon}{2E},\max\{b,\mu(s,b)\}\right)
\] 
for $b\geq\norm{x},\norm{y}$, we have
\[
\phi(x,J_ry)+\phi(J_ry,y)\leq\phi(x,y)+\varepsilon
\]
where
\[
E\geq \max\{2(\mu(r,b)+b),2rs^{-1}(\mu(r,b)+\mu(s,b))\}
\]
with $\mu$ as in Lemma \ref{lem:resMaj}.
\end{lemma}
\begin{proof}
It is known that
\begin{align*}
\phi(x,y)&=\phi(x,J_ry)+\phi(J_ry,y)+2\langle x-J_ry,JJ_ry-Jy\rangle\\
&=\phi(x,J_ry)+\phi(J_ry,y)+2r\langle x-J_ry,-A_ry\rangle.
\end{align*}
Now, we also immediately have (using monotonicity of $T$):
\begin{align*}
\langle x-J_ry,-A_ry\rangle&=\langle x-J_sx,-A_ry\rangle+\langle J_sx-J_ry,-A_ry\rangle\\
&=\langle x-J_sx,-A_ry\rangle+\langle J_sx-J_ry,s^{-1}(Jx-JJ_sx)-A_ry\rangle+\langle J_sx-J_ry,-s^{-1}(Jx-JJ_sx)\rangle\\
&\geq\langle x-J_sx,-A_ry\rangle+s^{-1}\langle J_sx-J_ry,JJ_sx-Jx\rangle\\
&\geq-\norm{x-J_sx}\norm{A_ry}-s^{-1}\norm{J_sx-J_ry}\norm{JJ_sx-Jx}\\
&\geq-\norm{x-J_sx}r^{-1}(\norm{J_ry}+\norm{y})-s^{-1}(\norm{J_sx}+\norm{J_ry})\norm{JJ_sx-Jx}.
\end{align*}
Thus, using the previous Lemma \ref{lem:resMaj}, we have 
\begin{align*}
\phi(x,y)&\geq\phi(x,J_ry)+\phi(J_ry,y)-2r\norm{x-J_sx}r^{-1}(\norm{J_ry}+\norm{y})-s^{-1}2r(\norm{J_sx}+\norm{J_ry})\norm{JJ_sx-Jx}\\
&\geq \phi(x,J_ry)+\phi(J_ry,y)-\norm{x-J_sx}2(\mu(r,\norm{y})+\norm{y})\\
&\qquad\qquad\qquad-2rs^{-1}(\mu(r,\norm{y})+\mu(s,\norm{x}))\norm{JJ_sx-Jx}\\
&\geq \phi(x,J_ry)+\phi(J_ry,y)-E\left(\norm{x-J_sx}+\norm{JJ_sx-Jx}\right)
\end{align*}
and therefore, for $x$ such that
\[
\norm{x-J_sx}\leq\omega\left(\frac{\varepsilon}{2E},\max\{b,\mu(s,b)\}\right),
\]
we get that
\[
\phi(x,y)\geq\phi(x,J_ry)+\phi(J_ry,y)-\varepsilon
\]
which is the claim.
\end{proof}

\begin{remark}\label{rem:Lem25quant}
If $\omega:\mathbb{N}\times\mathbb{N}\to\mathbb{N}$ is a modulus of uniform continuity of $J$ in the sense that
\[
\norm{z-w}\leq\frac{1}{\omega(k,b)+1}\to \norm{Jz-Jw}\leq\frac{1}{k+1}
\]
for all $k\in\mathbb{N}$ and all $z,w\in \overline{B}(0,b)$, then it holds that 
\[
\phi(x,J_ry)+\phi(J_ry,y)\leq\phi(x,y)+\frac{1}{k+1}
\]
for any $k,b\in\mathbb{N}$, $x,y\in X$ and $r,s>0$ whenever
\[
\norm{x-J_sx}\leq\frac{1}{\omega(2E(k+1)+1,B+1)+1}
\] 
for $B\geq\norm{x}$.
\end{remark}

Using this, we can give a quantitative version (in the sense of a modulus of uniform $\phi$-Fej\'er monotonicity as introduced before) of the ``Fej\'er monotonicity''-type result implicitly shown in the course of the proof of Theorem 3.1 in \cite{KKT2004}.

\begin{lemma}
Let $X$ be uniformly smooth and uniformly convex and let $T$ be maximally monotone. Let $d\in Tc$ with $C\geq\norm{c}$, $D\geq\norm{d}$ and let $\omega:\mathbb{N}\times\mathbb{N}\to\mathbb{N}$ be a modulus of uniform continuity of $J$. Then for all $r,n,m\in\mathbb{N}$:
\[
\forall z\in AF_{\chi(n,m,r)}\forall l\leq m\left( \phi(z,x_{n+l})\leq\phi(z,x_n)+\frac{1}{r+1}\right),
\]
where
\[
\chi(n,m,r)=\max\{n,\omega(2E^0_{n,m}((r+1)(m\dotdiv 1)\dotdiv 1+1)+1,M+1)\}
\]
as well as 
\[
E^0_{n,m}\geq \max\{2(\mu(\hat{r}_{n,m},M)+M),2\hat{r}_{n,m}r_n^{-1}(\mu(\hat{r}_{n,m},M)+\mu(r_n,M))\},
\]
where $\hat{r}_{n,m}=\max\{r_i\mid i\leq n+m\dotdiv  1\}$.
\end{lemma}
\begin{proof}
We get 
\[
\phi(z,x_{n+1})\leq\alpha_n\phi(z,x_n)+(1-\alpha_n)\phi(z,J_{r_n}x_n).
\]
as in the proof of Theorem 3.1 in \cite{KKT2004}. Now, since $z\in AF_{\chi(n,m,r)}$, we have
\[
\norm{z-J_{r_n}z}\leq\frac{1}{\omega(2E^0_{n,m}((r+1)(m\dotdiv 1)\dotdiv 1+1)+1,M+1)+1}
\]
and by Lemma \ref{lem:Lem25quant} (see also Remark \ref{rem:Lem25quant}), this implies
\[
\phi(z,J_{r_n}x_n)\leq\phi(z,x_n)+\frac{1}{(r+1)(m\dotdiv 1)\dotdiv 1+1}
\]
and by induction, combined with the above, we get
\begin{align*}
\phi(z,x_{n+l})&\leq\phi(z,x_{n})+\frac{1}{(r+1)(m\dotdiv 1)}\sum_{i=n}^{n+m\dotdiv 1}(1-\alpha_n)\\
&\leq \phi(z,x_{n})+\frac{m\dotdiv 1}{(r+1)(m\dotdiv 1)}\\
&\leq \phi(z,x_{n})+\frac{1}{r+1}
\end{align*}
for any $l\leq m$ since $\norm{z},\norm{x_i}\leq M$ for all $i$ as well as
\[
E^0_{n,m}\geq \max\{2(\mu(\hat{r}_{n,m},M)+M),2\hat{r}_{n,m}r_n^{-1}(\mu(\hat{r}_{n,m},M)+\mu(r_n,M))\}
\]
where $\hat{r}_{n,m}=\max\{r_i\mid i\leq n+m\dotdiv  1\}$.
\end{proof}

Next, we will obtain an approximate $F$-point bound from the asymptotic regularity result
\[
\phi(y_n,x_n)\to 0\text{ for }n\to\infty
\]
established in the proof of Theorem 3.1 from \cite{KKT2004}. For this, we first collect a few preliminary results.

\begin{lemma}
If $(a_n)\subseteq [0,b]$ and $a_{n+1}\leq a_n$ for all $n$, then
\[
\forall \varepsilon>0\exists n\leq 2\ceil*{\frac{b}{\varepsilon}}\left( a_{2n}-a_{2n+1}\leq\varepsilon\right).
\]
\end{lemma}
\begin{proof}
Suppose not. Then
\[
a_0>\varepsilon+a_2>\dots>2\ceil*{\frac{b}{\varepsilon}}\varepsilon+a_{2\ceil*{\frac{b}{\varepsilon}}}\geq b
\]
which is a contradiction.
\end{proof}

\begin{lemma}[\cite{KKT2004}]
Let $z\in\mathrm{zer}T$. Then for any $n$: $\phi(z,x_{n+1})\leq\phi(z,x_n)$ as well as
\[
\phi(y_n,x_n)\leq\frac{\phi(z,x_n)-\phi(z,x_{n+1})}{1-\alpha_n}.
\]
\end{lemma}

\begin{lemma}\label{lem:liminfratepre}
Let $z\in\mathrm{zer}T$, let $\bar{\alpha}$ be such that $1>\bar{\alpha}>\alpha_n$ for all $n$ and let $\bar{r}$ be such that $r_n\geq\bar{r}>0$ for all $n$. Let $b\geq\phi(z,x_0)$. Then
\[
\forall \varepsilon\exists n\leq 2\ceil*{\frac{b}{\varepsilon(1-\bar{\alpha})}} \left( \phi(J_{r_{2n}}x_{2n},x_{2n})\leq\varepsilon\right).
\]
\end{lemma}
\begin{proof}
The claim follows from the previous lemmas.
\end{proof}

\begin{lemma}
Let $z\in\mathrm{zer}T$, let $\bar{\alpha}$ be such that $1>\bar{\alpha}>\alpha_n$ for all $n$ and let $\bar{r}$ be such that $r_n\geq\bar{r}>0$ for all $n$. Let $b\geq\phi(z,x_0)$. Then
\[
\forall k\in\mathbb{N}\exists n\leq\Phi(k) \left( x_{2n}\in AF_k\right)
\]
with
\[
\Phi(k)=2\ceil*{\frac{(\lambda(\omega(2E^1_k(\lambda(k))+1,M+1))+1)b}{1-\bar{\alpha}}}
\]
as well as 
\[
E^1_k\geq \max\{2(\mu(\tilde{r}_k,M)+M),2\tilde{r}_k(\bar{r}^{-1}\mu(\tilde{r}_k,M)+\max\{(1+C)(\bar{r}^{-1}M+D)+C,\bar{r}^{-1}\}\}
\]
where $\tilde{r}_k=\max\{r_i\mid i\leq k\}$.
\end{lemma}
\begin{proof}
By Lemma \ref{lem:liminfratepre} and the properties of $\lambda$, we have
\[
\forall k\in\mathbb{N}\exists n\leq\Phi(k) \left( \norm{J_{r_{2n}}x_{2n}-x_{2n}}\leq\frac{1}{\omega(2E^1_k(\lambda(k))+1,M+1)+1}\right).
\]
By Lemma \ref{lem:Lem25quant}, we then in particular obtain that
\begin{align*}
\phi(x_{2n},J_{r_i}x_{2n})&\leq\phi(x_{2n},x_{2n})+\frac{1}{\lambda(k)+1}\\
&= \frac{1}{\lambda(k)+1}
\end{align*}
for all $i\leq k$. This in particular holds since
\begin{align*}
E^1_k &\geq \max\{2(\mu(\tilde{r}_k,M)+M),2\tilde{r}_k(\bar{r}^{-1}\mu(\tilde{r}_k,M)+\max\{(1+C)(\bar{r}^{-1}M+D)+C,\bar{r}^{-1}\}\}\\
&\geq \max\{2(\mu(\tilde{r}_k,M)+M),2\tilde{r}_k(\bar{r}^{-1}\mu(\tilde{r}_k,M)+r_{2n}^{-1}\max\{(1+C)(M+r_{2n}D)+C,1\}\}\\
&\geq \max\{2(\mu(\tilde{r}_k,M)+M),2\tilde{r}_k(\bar{r}^{-1}\mu(\tilde{r}_k,M)+r_{2n}^{-1}\mu(r_{2n},M))\}\\
&\geq \max\{2(\mu(\tilde{r}_k,M)+M),2\tilde{r}_kr_{2n}^{-1}(\mu(\tilde{r}_k,M)+\mu(r_{2n},M))\}
\end{align*}
where $\tilde{r}_k=\max\{r_i\mid i\leq k\}$. Thus $\norm{x_{2n}-J_{r_i}x_{2n}}\leq 1/(k+1)$ for all $i\leq k$.
\end{proof}

The combination of the previous modulus of uniform Fej\'er monotonicity with the above approximate $F$-point bound then yields the following rate of metastability as an application of Theorem \ref{thm:metaSingleDist}.

\begin{theorem}
Let $X$ be a uniformly convex and uniformly smooth Banach space and let $\eta:(0,2]\to (0,1]$ be a non-decreasing modulus of uniform convexity for $X$. Assume that $X_0$ as defined before is totally bounded with a modulus of total boundedness $\gamma$. Let $z\in\mathrm{zer}T$, let $\bar{\alpha}$ be such that $1>\bar{\alpha}>\alpha_n$ for all $n$ and let $\bar{r}$ be such that $r_n\geq\bar{r}>0$ for all $n$. Let $b\geq\phi(z,x_0)$. Let $d\in Tc$ with $C\geq\norm{c}$, $D\geq\norm{d}$ and let $\omega:\mathbb{N}\times\mathbb{N}\to\mathbb{N}$ be a modulus of uniform continuity of $J$. Then $(x_n)$ is Cauchy and for all $k\in\mathbb{N}$ and all $g:\mathbb{N}\to\mathbb{N}$:
\[
\exists N\leq\tilde{\Psi}(\lambda(k),g)\forall i,j\in[N;N+g(N)]\left( \norm{x_i-x_j}\leq\frac{1}{k+1}\text{ and }x_i\in AF_k\right)
\]
with $\tilde{\Psi}(k,g)=\Psi(k_0,g)$ where
\[
k_0=\max\left\{k,\lambda(\omega_F(k))\right\}
\]
where $\Psi(k,g)=2\Psi_0(P,k,g)$ with $P=\gamma(\Lambda(16\theta(k)+15))$ and $\Psi_0$ defined by
\[
\begin{cases}
\Psi_0(0,k,g)=0,\\
\Psi_0(n+1,k,g)=\Phi^M(\eta^M_k(\Psi_0(n,k,g),2\theta(k)+1)),
\end{cases}
\]
where
\begin{gather*}
\Lambda(k)=\ceil*{\frac{k+1}{16L^{-1}(M+1)}}\dotdiv 1,\\
\lambda(k)=\ceil*{\frac{L}{(M+1)^2}\left(\eta\left(\frac{1}{4(M+1)(k+1)}\right)\right)^{-1}}\dotdiv 1,\\
\theta(k)=\ceil*{\frac{L}{(M+1)^2}\left(\eta\left(\frac{1}{8(M+1)(k+1)}\right)\right)^{-1}}\dotdiv 1,
\end{gather*}
as well as
\[
\Phi(k)=2\ceil*{\frac{(\lambda(\omega(2E^1_k(\lambda(k))+1,M+1))+1)b}{1-\bar{\alpha}}}
\]
and $\Phi^M(k)=\max\{\Phi(j)\mid j\leq k\}$ and where
\[
\eta_k(n,r)=\max\{\chi'_{k,g}(n,r),\zeta_g(n,r),\delta(n,r)\}
\]
with
\begin{gather*}
\chi(n,m,r)=\max\{n,\omega(2E^0_{n,m}((r+1)(m\dotdiv 1)\dotdiv 1+1)+1,M+1)\},\\
\chi_k(n,m,r)=\max\{\delta_F(k),\chi(n,m,r)\},\\
\chi'_k(n,m,r)=\chi_k(2n,2m,r),\\
\zeta(n,m,r)=\chi(2n,2m+1,r),
\end{gather*}
as well as
\begin{gather*}
\delta(n,r)=\chi\left(f(n),n-f(n),4r+3\right)\},\\
\chi'_{k,g}(n,r)=\chi'_k\left(n,\floor*{\frac{g(2n)}{2}},r\right),\\
\zeta_g(n,r)=\zeta\left(n,\floor*{\frac{g(2n)}{2}},r\right).
\end{gather*}
for
\begin{gather*}
E^0_{n,m}\geq \max\{2(\mu(\hat{r}_{n,m},M)+M),2\hat{r}_{n,m}r_n^{-1}(\mu(\hat{r}_{n,m},M)+\mu(r_n,M))\},\\
E^1_k\geq \max\{2(\mu(\tilde{r}_k,M)+M),2\tilde{r}_k(\bar{r}^{-1}\mu(\tilde{r}_k,M)+\max\{(1+C)(\bar{r}^{-1}M+D)+C,\bar{r}^{-1}\}\},
\end{gather*}
where $\tilde{r}_k=\max\{r_i\mid i\leq k\}$ and $\hat{r}_{n,m}=\max\{r_i\mid i\leq n+m\dotdiv  1\}$ as well as 
\[
\mu(R,b)=\max\{(1+C)(b+RD)+C,1\}.
\]
\end{theorem}

If we have a modulus of regularity for the operator in the sense of the following definition, then an application of Theorem \ref{thm:modregthm} even allows for the derivation of a rate of convergence.

\begin{definition}
Let $S\subseteq X$ be a set. A function $\rho:\mathbb{R}^*_+\to\mathbb{R}^*_+$ is a \emph{modulus of regularity for the (generalized) proximal point algorithm w.r.t.\ $S$} if for all $\varepsilon>0$ and $x\in S$:
\[
\norm{x-J_1x} <\rho(\varepsilon)\text{ implies }\mathrm{dist}(x,\mathrm{Fix}\,J_1)<\varepsilon.
\]
\end{definition}

\begin{theorem}
Let $X$ be a uniformly convex and uniformly smooth Banach space and let $\eta:(0,2]\to (0,1]$ be a modulus of uniform convexity for $X$. Let $\omega(\varepsilon,b)\leq\varepsilon$ be a modulus of uniform continuity of $J$ on bounded subsets, i.e.
\[
\forall z,w\in X,b\in\mathbb{N},\varepsilon>0\left(\norm{z},\norm{w}\leq b\land 
\norm{z-w}\leq\omega(\varepsilon,b)\to \norm{Jz-Jw}\leq\varepsilon\right).
\]
Let $X_0$ be defined as before. Let $z\in\mathrm{zer}T$ and let $b\geq \phi(z,x_0)$. Let $\bar{\alpha}$ be such that $1>\bar{\alpha}>\alpha_n$ for all $n$ and let $\bar{r}$ be such that $r_n\geq\bar{r}>0$ for all $n$. Let $\rho$ be a modulus of regularity for the proximal point algorithm w.r.t.\ $X_0$. Then $(x_n)$ is Cauchy and
\[
\forall \delta>0\forall n,m\geq\mu(\delta)(\norm{x_n-x_m}<\delta)
\]
with 
\[
\mu(\delta)=2\tau\left(\rho'\left(\frac{ \theta(\lambda(\delta))}{2}\right)\right)
\]
and where
\begin{gather*}
\tilde{\Lambda}(\varepsilon)=\frac{\varepsilon}{16L^{-1}(M+1)},\\
\tilde{\lambda}(\varepsilon)=\frac{(M+1)^2}{L}\eta\left(\frac{\varepsilon}{4b(M+1)}\right),\\
\theta(\varepsilon)=\tilde{\lambda}\left(\varepsilon/2\right),\\
\rho'(\varepsilon)=\rho(\tilde{\Lambda}(\varepsilon/2)).
\end{gather*}
and where
\[
\tau(\varepsilon)=2\ceil*{\frac{b}{\tilde{\lambda}\left(\omega\left(\frac{\tilde{\lambda}(\varepsilon)}{2E},M\right)\right)(1-\bar{\alpha})}}
\]
with
\[
E\geq \max\{2(\mu(1,M)+M),2(\bar{r}^{-1}\mu(1,M)+\max\{(1+C)(\bar{r}^{-1}M+D)+C,\bar{r}^{-1}\}\}
\]
as well as $\mu(R,b)=\max\{(1+C)(b+RD)+C,1\}$.
\end{theorem}
\begin{proof}
By Lemma \ref{lem:liminfratepre}, we have 
\[
\forall \varepsilon\exists n\leq 2\ceil*{\frac{b}{\varepsilon(1-\bar{\alpha})}} \left( \phi(J_{r_{2n}}x_{2n},x_{2n})\leq\varepsilon\right).
\]
Thus, we get that there exists an $n\leq\tau(\varepsilon)$ such that
\[
\norm{x_{2n}-J_{r_{2n}}x_{2n}}\leq\omega\left(\frac{\tilde{\lambda}(\varepsilon)}{2E},M\right).
\]
By Lemma \ref{lem:Lem25quant}, we have
\[
\phi(x_{2n},J_1x_{2n})\leq\phi(x_{2n},x_{2n})+\tilde{\lambda}(\varepsilon)=\tilde{\lambda}(\varepsilon)
\]
and we thus have $\norm{x_{2n}-J_1x_{2n}}\leq\varepsilon$. For this, note that
\begin{align*}
E&\geq \max\{2(\mu(1,M)+M),2(\bar{r}^{-1}\mu(1,M)+\max\{(1+C)(\bar{r}^{-1}M+D)+C,\bar{r}^{-1}\}\}\\
&\geq \max\{2(\mu(1,M)+M),2(\bar{r}^{-1}\mu(1,M)+r_{2n}^{-1}\max\{(1+C)(M+r_{2n}D)+C,1\}\}\\
&\geq \max\{2(\mu(1,M)+M),2(\bar{r}^{-1}\mu(1,M)+r_{2n}^{-1}\mu(r_{2n},M))\}\\
&\geq \max\{2(\mu(1,\norm{x_{2n}})+\norm{x_{2n}}),2 r_{2n}^{-1}(\mu(1,\norm{x_{2n}})+\mu(r_{2n},\norm{x_{2n}}))\}.
\end{align*}
Thus $\tau$ confines to the properties of Theorem \ref{thm:modregthm} and this result consequently implies the present theorem.
\end{proof}

Similar to before, as already discussed in Remark \ref{rem:noComp}, also here the strong convergence in particular holds without any additional compactness assumptions in the context of these moduli of regularity.

\medskip

Again, we refer to \cite{KLAN2019} for the discussion of examples of concrete moduli of regularity for resolvents of monotone operators in Hilbert spaces which can be easily transformed into corresponding moduli of regularity for resolvents of monotone operators in Banach spaces in the above sense.\\

{\bf Acknowledgments}
\noindent I want to thank Ulrich Kohlenbach for suggesting the algorithms presented in \cite{IH2019,MP2015} as interesting case studies as well as for suggesting to consider potential abstract results as presented in the first part of this paper, generalizing the previous works \cite{KLN2018,KLAN2019} to the setting of partial Fej\'er monotonicity. For another, I want to emphasize that Lemmas \ref{lem:MuPeng1819Av}, \ref{lem:MuPengTheorem1iAv} and \ref{lem:modulusUnifFejer} are generalizations of quantitative versions due to him of the analogous results from \cite{MP2015} that he communicated to me in personal correspondence. Further, I want to thank Thomas Powell and Morenikeji Neri for many remarks that led to simplifications in the assumptions of various theorems in this paper. Lastly, I want to thank the anonymous referees for the valuable comments that improved the paper at various places.

{\bf Funding and/or Conflicts of interests}
\noindent This work was supported by the `Deutsche Forschungs\-gemein\-schaft' Project DFG KO 1737/6-2. The author has no relevant financial or non-financial interests to disclose.

\bibliographystyle{plain}
\bibliography{ref}

\end{document}